\documentclass[11pt,reqno]{amsart} 
\usepackage{preamble, fullpage, pdflscape, bbm}
\usepackage{adjustbox}

\theoremstyle{plain}
\newtheorem*{theorem}{Theorem}
\newtheorem*{lemma}{Lemma}
\newtheorem*{proposition}{Proposition}
\newtheorem*{example}{Example}
\newtheorem*{definition}{Definition}
\newtheorem*{corollary}{Corollary}
\title{Gaudin Algebras, RSK and Calogero-Moser Cells in Type A}
\thanks{I.G. has been supported by EPSRC grants EP/R034826/1 and EP/G007632/1.}
\author{Adrien Brochier}
\address{Universit\'e de Paris, Sorbonne Universit\'e, CNRS, Institut de
Math\'ematiques de Jussieu-Paris Rive Gauche,
F-75013 Paris, France
}
\email{adrien.brochier@imj-prg.fr}
\author{Iain Gordon}
\address{School of Mathematics, University of Edinburgh, Edinburgh, U.K.}
\email{igordon@ed.ac.uk}
\author{Noah White}
\address{Mathematical Sciences Institute, Australian National University, Canberra, Australia}
\email{noah.white@anu.edu.au}
\begin{document}
\maketitle
\begin{abstract}
    We study the spectrum of a family of algebras, the inhomogeneous Gaudin algebras, acting on the $n$-fold tensor representation $\mathbb{C}[x_1, \ldots, x_r]^{\ot n}$ of the Lie algebra $\gl_r$. We use the work of Halacheva-Kamnitzer-Rybnikov-Weekes to demonstrate that the Robinson-Schensted-Knuth correspondence describes the behaviour of the spectrum as we move along special paths in the family. We apply the work of Mukhin-Tarasov-Varchenko, which proves that the rational Calogero-Moser phase space can be realised as a part of this spectrum, to relate this to behaviour at $t=0$ of rational Cherednik algebras of $\mathfrak{S}_n$. As a result, we confirm for symmetric groups a conjecture of Bonnaf\'{e}-Rouquier which proposes an equality between the Calogero-Moser cells they defined and the well-known Kazhdan-Lusztig cells. 
\end{abstract}
\section{Introduction}
\label{sec:introduction}
\subsection{} 
Let $\glr = {\rm Mat}_r( \mathbb{C})$ be the general Lie algebra. Any tensor product of locally finite dimensional $\glr$-representations, $M = V_1 \ox V_2 \ox \cdots \ox V_n$, carries an action of an inhomogeneous Gaudin algebra,  $\mathcal{A}(\underline{z}; \underline{q})$. This algebra depends on two sets of parameters $\underline{z} = (z_1, \ldots , z_n)\in \mathbb{C}^n$ and $\underline{q} = (q_1, \ldots , q_r) \in \mathbb{C}^r$. It is commutative. Works of several authors prove that for many $M$ and general values of the parameters $\underline{z}$ and $\underline{q}$ its action has a simple spectrum. We will denote its spectrum by $\mathcal{E}_{\underline{z}, \underline{q}}(M)$. 
\subsection{} We are interested in the behaviour of the sets $\mathcal{E}_{\underline{z}, \underline{q}}(M)$ as we vary the parameters $\underline{z}$ and $\underline{q}$. We shall investigate the case $M = \mathbb{C}[x_1, \ldots , x_r]^{\otimes n}$, where $\glr$ acts on $\mathbb{C}[x_1, \ldots , x_r]$ via differential operators $E_{ij}\mapsto x_i\partial_j$. For the rest of the introduction we fix $M$ as this representation, and denote $\mathcal{E}_{\underline{z}, \underline{q}}(M)$ by $\mathcal{E}_{\underline{z}, \underline{q}}$.
\subsection{} We shall consider tuples $\underline{z}\in \Rdel^n = \{ (z_1, \ldots , z_n) \in \mathbb{R}^n: z_1<z_2<\cdots <z_n\}$ and $\underline{q}\in\Rdel^r = \{ (q_1, \ldots , q_r) \in \mathbb{R}^r: q_1<q_2<\cdots <q_r\}$. We will prove, with appropriate definitions for limits, that:
\begin{itemize}
    \item as $\underline{z} \rightarrow \underline{\infty}$, the set $\mathcal{E}_{\underline{\infty}, \underline{q}}$ identifies with ${\rm Mat}_{r\times n}(\mathbb{N})$, the set of $r$-by-$n$ matrices whose entries are non-negative integers;
   
    \item as $\underline{z}\rightarrow \underline{0}$ and then $\underline{q} \rightarrow \underline{0}$, the set $\mathcal{E}_{\underline{0}, \underline{0}}$ identifies with $\bigsqcup_{\lambda \in \Part} \SSYT_n(\lambda) \times \SSYT_r(\lambda)$ where  $\SSYT_r(\lambda)$ is the set of semistandard Young tableaux of shape $\lambda$ with entries from $\{1, \ldots, r\}$ and similarly for $\SSYT_n(\lambda)$;
    \item tracking the sets $\mathcal{E}_{\underline{z}, \underline{q}}$ through the process of moving $(\underline{z}, \underline{q})$ in (an extension of) $\Rdel^n \times \Rdel^r$ from the first limit to the second limit induces the Robinson-Schensted-Knuth correspondence $$\Mat_{r \times n}(\mathbb{N}) \longrightarrow \bigsqcup_{\lambda \in \Part} \SSYT_n(\lambda) \times \SSYT_r(\lambda).$$  
\end{itemize}
\subsection{} Our proof of these results relies critically on \cite{HKRW} which endows the sets $\mathcal{E}_{\underline{z},\underline{q}}(M)$ with the structure of a $\glr$-crystal. That the RSK-correspondence appears is then a common theme in the theory of crystals.
\subsection{} Our interest in the above result stems from an application of a special case to confirm conjectures of Bonnaf\'{e}-Rouquier, \cite[Conjecture L and LR]{BonRouq}, in the theory of rational Cherednik algebras of $\mathfrak{S}_n$ which we now explain.
 To each complex reflection group $(W, \mathfrak{h})$ there is a family of rational Cherednik algebras associated, \cite{Etingof:2002bc}. These algebras depend on a pair of parameters, usually denoted by $t$ and $c$. When $t=0$, the rational Cherednik algebras have a large centre whose geometry controls much of their representation theory. The spectrum of this centre, a generalised Calogero-Moser space, depends on the parameter $c$ and is a ramified covering of the affine space $\mathfrak{h}/W \times \mathfrak{h}^*/W$. 
Bonnaf\'{e} and Rouquier, \cite{BonRouq},  have used the Galois theory of this covering to define partitions of the elements of $W$ into (left, right, two-sided) Calogero-Moser ``cells'', depending on $c$. When $(W, \mathfrak{h})$ is a Coxeter group it is conjectured that with appropriate choices these cells agree with the Kazhdan-Lusztig cells of $W$, important objects in Lie theoretic representation theory and algebraic combinatorics. Thus, conjecturally, Calogero-Moser cells generalise the theory of Kazhdan-Lusztig cells from Coxeter groups to all complex reflection groups. 
In \cite{BonRouq} the Bonnaf\'{e}-Rouquier conjecture is proved for rank 2 Coxeter groups. 
\subsection{} At $t=0$ the spectrum of the centre of rational Cherednik algebra of $\mathfrak{S}_n$ (for $c\neq 0$) is isomorphic to classical rational Calegero-Moser phase space $$\CM_n = \{ (Z, Y) \in {\rm Mat}_n (\mathbb{C}) \times {\rm Mat}_n (\mathbb{C}): [Z,Y] + \id = \text{rank 1 matrix} \}/ PGL_n(\mathbb{C}).$$ Sending the pair of matrices $(Z,Y)$ to their eigenvalues produces a ramified covering $$\Upsilon: \CM_n \rightarrow \mathbb{C}^n/\mathfrak{S}_n \times \mathbb{C}^n/\mathfrak{S}_n.$$
Any $n$-tuple of pairwise distinct complex numbers, $\underline{z} = (z_1, \ldots , z_n) \in \Creg^n$ and $n$-tuple  $\underline{p} = (p_1, \ldots , p_n)\in \mathbb{C}^n$ give rise to a point in $\CM_n$: $$ Z = \begin{pmatrix} z_1 & 0 & \cdots & 0 \\ 0 & z_2 & \cdots & 0 \\ \vdots & \vdots & & \vdots \\ 0 & 0 & \cdots & z_n  \end{pmatrix}, \, Y = \begin{pmatrix} p_1 & (z_1-z_2)^{-1} & \cdots & (z_1-z_n)^{-1} \\ (z_2-z_1)^{-1} & p_2 & \cdots & (z_2-z_n)^{-1} \\ \vdots & \vdots & & \vdots \\ (z_n-z_1)^{-1} & (z_n-z_2)^{-1} & \cdots & p_n  \end{pmatrix}.$$ 
In this description $\Upsilon (Z,Y) = ([\underline{ z}], [\underline{q}])$ where the unordered eigenvalues of the matrix $Y$ depend on $\underline{p}$ and $\uz$ and are denoted by $[\underline{q}]$.
\subsection{} In this language the Bonnaf\'{e}-Rouquier conjecture can be stated in terms of the behaviour of $\underline{z}= (z_1, \ldots , z_n) \in \Rdel^n$ and $\underline{q} = (q_1, \ldots , q_n) \in \Rdel^n$. For appropriate large values of $\underline{z}$, the elements in the fibre of $\Upsilon$ can be identified with the symmetric group, since the matrices $Y$ above $[\underline{q}]$ are close to diagonal matrices with distinct entries. The fibre above $[\underline{z}] = [\underline{0}]$ can be identified with the standard Young tableaux of $n$, \cite[Theorem 1.5]{Mukhin:2012ee}. Tracking the elements in the fibre as $\underline{z}$ moves from infinity to zero sees them coalesce to tableaux. The conjecture is that these collisions are determined by the Robinson-Schensted algorithm, and in particular that the element in the fibre corresponding to $w \in \mathfrak{S}_n$ tracks to the tableau given by the $P$-symbol of $w$. 
\subsection{} To confirm this we use the work \cite{Mukhin:2014ga} which identifies the fibres of $\Upsilon$ with a part of the spectrum $\mathcal{E}_{\underline{z}, \underline{q}}(V_{\varpi_1}^{\otimes n})$ of the inhomogeneous Gaudin algebras for $\mathfrak{gl}_n$. Since $V_{\varpi_1}$ is a summand of $\mathbb{C}[x_1, \ldots , x_n]$,  $V_{\varpi_1}^{\otimes n}$ is a summand of $\mathbb{C}[x_1, \ldots , x_n]^{\otimes n}$ and we are in the situation described above in the introduction, with $r=n$. 
As a result we can use the limiting behaviour of the spectrum $\mathcal{E}_{\underline{z}, \underline{q}}$ to interpret the behaviour of the fibres of $\Upsilon$ in terms of the Robinson-Schensted correspondence. This allows us to confirm Bonnaf\'{e}-Rouquier conjecture for $\mathfrak{S}_n$. 
\subsection{} This paper is organised as follows. In the Sections 2 and 3 we recall a variety of constructions which we will require: parallel transport of representations of algebras, the RSK correspondence, inhomogeneous Gaudin algebras. We state our main theorem in Section 4. Sections 5--7 then prove the theorem. In Section 8 we apply this to confirm the conjecture of Bonnaf\'{e}-Rouquier.  
\subsection{Acknowledgements} We thank Cedric Bonnaf\'e, Leonid Rybnikov and Sasha Veselov for useful conversations. This paper was written while the authors visited several institutions. In particular we would like to thank the Hausdorff Research Institute for Mathematics, ETH Z\"urich, FAU Erlangen-N\"urnberg.
\section{Setup}
\label{sec:setup}
In this section we introduce notation that we will use throughout the paper.
\subsection{Spectra of Commutative Algebras} Let \( V \) be a \( k \)-dimensional vector space. Assume that \( \mathcal{A} \) is a commutative algebra that acts on \( V \), meaning that we have a mapping \( \mathcal{A}\longrightarrow  \End(V) \). Since \( \mathcal{A} \) is commutative there will exist a set of distinct algebraic characters \( \chi_1, \ldots , \chi_{\ell}  \) such that for each \( i \) the generalised eigenspace \[ V_i = \left\{ v \in V \;\middle|\; (a-\chi_i(a))^k\cdot v = 0 \text{ for all } a \in \mathcal{A} \right\} \] is non-zero. It follows that \( V = V_1\oplus \cdots \oplus V_{\ell} \) is an \( \mathcal{A} \)-stable decomposition. We denote the set of all generalised eigenspaces for \( \mathcal{A} \) acting on \( V \) by \[ \mathcal{E}_{\mathcal{A}}(V) = \{ V_1, \ldots, V_{\ell}\}. \]
 We say that \( \mathcal{A} \) acts with a {simple spectrum} on \( V \) if for each \( i \) the space \( V_i \) above is one-dimensional. In particular this means that $\ell$, the number of characters, equals \( k \). 
\subsection{}
\label{continuation}
Suppose $U$ is an algebra and we have an algebra map \( U \longrightarrow \End(V) \). A family of commutative subalgebras of \( U \) parametrized by a space \( X \), denoted $\{ \mathcal{A}(x): x\in X\}$, gives us a family of subalgebras acting on \( V \). Typically the images of these subalgebras in \( \End(V) \) are not all of the same dimension nor do they all act semisimply. Nonetheless we will be interested in two different situations. 
\begin{enumerate}
    \item We will restrict \( X \) to some topological subspace \( Y \subseteq X \) such that for each \( y \in Y \), \( \mathcal{A}(y) \) acts with simple spectrum on \( V \). Then we get a covering space \[ \mathcal{E}(V )\longrightarrow Y \] whose fibres
are \( \mathcal{E}_{\mathcal{A}(y)}(V) \). In particular, if we have a path \( \gamma : [0, 1] \longrightarrow Y \) we get a parallel transport
map \( p : \mathcal{E}_{\mathcal{A}(\gamma(0))}(V ) \longrightarrow \mathcal{E}_{\mathcal{A}(\gamma(1))}(V ) \).
\item We will restrict \( X \) to some topological subspace \( Y \subseteq X \) such that for each \( y \in Y \), \( \mathcal{A}(y) \) acts with simple spectrum on \( V \) with the exception of a single distinguished point \( \tilde{y}\in Y \) where \( \mathcal{A}(\tilde{y}) \) acts semisimply, but without a simple spectrum.  Then we get a branched covering space \[ \mathcal{E}(V )\longrightarrow Y. \] In this case, if we have a path \( \gamma : [0, 1] \longrightarrow Y \) with \( \gamma(t) = \tilde{y} \) if and only if \( t=1 \), then we get a degenerated parallel transport
map \( p : \mathcal{E}_{\mathcal{A}(\gamma(0))}(V ) \longrightarrow \mathcal{E}_{\mathcal{A}(\gamma(1))}(V ) \). 
\end{enumerate}
\subsection{Tensor Embeddings}
Let \( U \) be an algebra and \( x \in U \). Let $n\in \mathbb{N}$ and $1\leq a \leq n$. We have an embedding
\[ \iota_a: U \longrightarrow U^{\ox n}, \quad  x \mapsto \id^{\ox a-1}\ox x \ox \id^{\ox n-a} \] which places \( x \) in the \( a^{\text{th}} \) tensorand. We use the notation \( x^{(a)} = \iota_a(x) \) for this. If \( \underline{a}=(a_1,a_2,\ldots,a_k) \) is a sequence of distinct integers between \( 1 \) and \( n \), we use \( \iota_{\underline{a}}:U^{\ox k} \longrightarrow U^{\ox n} \) to denote the map that embeds the \( i^{\text{th}} \) factor into the \( a_i^{\text{th}} \) factor, i.e. \( \iota_{\underline{a}}(x_1\ox x_2 \ox \cdots\ox x_k) = \prod_{i=1}^k x_i^{(a_i)} \). 
\subsection{} If \( U \) is a bialgebra with coproduct \( \Delta:U \longrightarrow U\ox U \), we use \( \Delta^{n} \) to denote the map \( U \longrightarrow U^{\ox n} \) defined inductively by \( \Delta^{n} = (\Delta^{n-1}\ox \id) \circ \Delta \) with \( \Delta^{1}=\id \). Note this means \( \Delta^2=\Delta \). For \( \underline{a} =(a_1,a_2,\ldots,a_k) \) as above, let \( \#\underline{a}=k \). If \( \underline{a}^1,\underline{a}^2,\ldots,\underline{a}^m \) are sequences of distinct integers between \( 1 \) and \( n \), such that the underlying sets are disjoint, we let \( \Delta^{\underline{a}^1,\underline{a}^2,\ldots,\underline{a}^m}: U^{\ox m} \longrightarrow U^{\ox n} \) be the map defined by \( \Delta^{\underline{a}^1,\underline{a}^2,\ldots,\underline{a}^m} = \prod_{i=1}^m \iota_{\underline{a}^i} \circ \Delta^{\# \underline{a}^i} \). If \( A \subseteq U^{\ox m} \) is a subalgebra we also use the notation \( A^{\underline{a}^1,\underline{a}^2,\ldots,\underline{a}^m} = \Delta^{\underline{a}^1,\underline{a}^2,\ldots,\underline{a}^m}(A) \). For example, if \( A \subseteq U^{\ox 2} \) we think of \( A^{(23)(1)} \subseteq U^{\ox 3} \) as the algebra obtained from \( A \) by spreading its first leg over the second and third tensor factors, and placing the second leg into the first tensor factor.
\subsection{The Robinson-Schensted-Knuth correspondence}
\label{sec:RSK} A thorough account can be found in \cite[Chapter 4]{fulton_Young}. Let \( \Mat_{r \times n}(\NN) \) be the set of \( r \times n \) matrices with nonnegative integer entries. Let \( k,t \in \mathbb{N}\). We set \( \Part(k) \) to be the set of partitions of \( k \) and the set of all partitions is denoted \( \Part = \bigsqcup_{k \ge 0} \Part(k) \), and we let \( \Part_{\le t} \) be the set of partitions with at most \( t \) parts.
\subsection{} We define a map
\[ \RSK : \Mat_{r \times n}(\NN) \longrightarrow \bigsqcup_{\lambda \in \Part} \SSYT_n(\lambda) \times \SSYT_r(\lambda). \]
via the following algorithm. For a matrix \( A = (a_{ij}) \in \Mat_{r \times n}(\NN) \) we form a sequence of pairs of integers \( (i_1,j_1),(i_2,j_2),\ldots,(i_k,j_k) \) where \( k \) is the sum of the entries of \( A \). The number of times the pair \( (i,j) \) appears in the sequence is \( a_{ij} \). The sequence is ordered in lexiographic order, giving preference to the first entry in each pair.  For example,
\[ A=
  \begin{pmatrix}
    0&2&1\\1&0&1
  \end{pmatrix} \mapsto (1,2), (1,2), (1,3), (2,1), (2,3).
\]
Now we {\it insert} the sequence \( j_1,j_2,\ldots,j_k \) into a tableau to form a semistandard tableau \( P(A) \) and a standard tableau \( Q'(A) \) (the recording tableau). To create the semistandard tableau \( Q(A) \) we  replace the number \( m \) in \( Q'(A) \) by \( i_m \). So in the above example
\[ P(A) = \young(1233,2), \,\, Q'(A) = \young(1235,4), \,\, Q(A)= \young(1112,2) . \]
By definition \( \RSK (A) = (P(A), Q(A)) \).
\subsection{} Transposing the matrix $A$ coincides with swapping the order of $P$ and $Q$. 
\begin{theorem}[{\cite[pp.40-41]{fulton_Young}}]
  \label{thm:RSK}
The map \( \RSK \) is a bijection. Moreover \( \RSK(A^t) = (Q(A),P(A)) \). 
\end{theorem}
\subsection{Crystals} 
We let \( \SSYT_r(\lambda) \) be the set of semistandard tableaux of shape \( \lambda \in \Part \) with entries in \( \{1,2,\ldots,r\} \). Recall that the set \( \SSYT_r(\lambda) \) carries a natural \( \gl_r \)-crystal structure. 
The set \( \Mat_{r \times n}(\NN) \) is given the structure of a \( \gl_r \)-crystal as follows. Let \( \NN^r(k) \) be the set of vectors in \( \NN^r \) with the sum of entries equal to \( k \). We can identify \( \NN^r(k) \) with the natural basis of monomials in the irreducible $\gl_r$-representation \( \mathbb{C}[x_1,x_2, \ldots , x_r]_k \), 
giving \( \NN^r(k) \) the structure of a \( \gl_r \)-crystal.
If \( \underline{k}=(k_1,k_2,\ldots,k_n) \in \NN^n \), then let \( \Mat_{r \times n}(\NN, \underline{k}) \) be the set of matrices with column sums \( \underline{k} \).
This identifies \( \Mat_{r \times n}(\NN,\uk)\) with \(\prod_{i=1}^n \NN^r(k_i) \), and so gives \( \Mat_{r \times n}(\NN, \underline{k}) \) a \( \gl_r \)-crystal structure as a tensor product of crystals. 
Finally, \( \Mat_{r \times n}(\NN) = \bigsqcup_{\underline{k} \in \NN^n} \Mat_{r \times n}(\NN, \underline{k}) \) is the direct sum (i.e. disjoint union) of crystals.
\begin{proposition}[{\cite[Corollary~9.2]{BumpSchilling}}]
  \label{prp:RSK-crystal-iso}
The map \( \RSK \) is an isomorphism of \( \gl_r \)-crystals, where \( \SSYT_n(\lambda) \times \SSYT_r(\lambda) \) is taken to be \( \# \SSYT_n(\lambda) \) many copies of the crystal \( \SSYT_r(\lambda) \).
\end{proposition}
\subsection{} The following is an important property of the RSK correspondence which we will exploit later. If \( 1 \le i \le r \)  and  \( S \in \SSYT_r(\lambda) \) we set \( S|_i \) to be the tableau obtained by removing all boxes containing numbers strictly larger than \( i \). 
\begin{proposition}
  \label{prp:rsk-rigidity}
Suppose \( f:\Mat_{r \times n}(\NN) \longrightarrow \bigsqcup_{\lambda \in \Part} \SSYT_n(\lambda) \times \SSYT_r(\lambda) \), $f(A) = (S(A),T(A))$, is an isomorphism of \( \gl_r \)-crystals. If, for every \( A \in \Mat_{r \times n}(\NN) \), $f$ has the property that \( S(A)|_{n-1} = P(A)|_{n-1} \), then \( f = \RSK \).
\end{proposition}
\begin{proof}
  The fact that \( f \) and \( \RSK \) are both crystal isomorphisms means that \( Q(A)=T(A) \).
  For \( T \in \SSYT_n(\lambda) \), define \( \sh(T)=\lambda \) and observe therefore that
  \[ \sh(S(A))=\sh(T(A))=\sh(Q(A))=\sh(P(A)). \]
  Thus we know that \( S(A) \) and \( P(A) \) are semistandard tableaux of the same shape. By hypothesis they are identical semistandard tableaux once all boxes containing \( n \) have been removed. This means \( S(A)=P(A) \).
\end{proof}
\section{Inhomogeneous Gaudin Algebras}
\label{sec:bethe-algebras}
We will introduce the algebras whose spectrum we will be interested in studying and some of their limits.
\subsection{Definitions} Let \( \hat{\gl_r}_- = t^{-1}\gl_r[t^{-1}] \). For any non-zero complex number \( w \), there is an evaluation map \( \hat{\gl_r}_- \rightarrow \gl_r \) which sends \( g\otimes t^{-s} \) to \( gw^{-s} \). It is a Lie algebra homomorphism and induces a morphism of algebras \( \phi_w:U(\hat{\gl_r}_-)\longrightarrow U(\gl_r) \). For any \( \underline{z} = (z_1, \ldots , z_n) \in \CC^n \) and \( w\in \CC\setminus \{ z_1, \ldots, z_n \} \), there is a Lie algebra homomorphism:
\[ \hat{\gl_r}_- \longrightarrow \gl_r^{\oplus n}, \qquad g\otimes t^{-s} \mapsto (g(w-z_1)^{-s}, g(w-z_2)^{-s}, \ldots , g(w-z_n)^{-s}). \] This induces an algebra morphism \[ \phi_{w}(\underline{z})=\phi_{w-z_1} \ox \phi_{w-z_2}\ox \cdots \ox \phi_{w-z_n} \circ \Delta^{n}: U(\hat{\gl_r}_-) \longrightarrow U(\gl_r^{\oplus n}) = U(\gl_r)^{\otimes n}. \]
There is also an evaluation at \( \infty \) mapping, \( \hat{\gl_r}_- \rightarrow (\gl_r)_{\sf ab} \), given by extracting the coefficient of \( t^{-1} \). This Lie algebra homomorphism gives rise to an algebra morphism \[ \phi_{\infty}: U(\hat{\gl_r}_-) \longrightarrow U((\gl_r)_{\sf ab}) = S(\gl_r). \]
\subsection{} Now fix \( \underline{q} = (q_1,\ldots, q_r) \in \CC^r \) which we identify with an element of \( \diag(\gl_r) \), the Cartan subalgebra of \( \gl_r \). Using the trace form we can identify \( \underline{q} \) with a functional on \( \gl_r \) which extends to a character of the symmetric algebra \( S(\gl_r) \), denoted by \( \chi_{\underline{q}} \). We then define the mapping 
\[ \phi_{w}(\underline{z};\underline{q}) = (\id^{\otimes n}\otimes \chi_{\underline{q}}) \circ (\phi_{w}(\uz)\otimes \phi_{\infty})\circ \Delta : U(\hat{\gl_r}_-) \longrightarrow U(\gl_r)^{\otimes n}. \]  
Alternatively,
\[ \phi_w(\uz;\uq) = \phi_{w-z_1}\ox\phi_{w-z_2}\ox\cdots\ox\phi_{w-z_n}\ox(\chi_{\uq} \circ \phi_{\infty}) \circ \Delta^{n+1}. \]
\subsection{}
\label{sec:equal-params-gaudin}
We now describe how this map behaves when some of the parameters \( \uz\in \mathbb{C}^n \) coincide. Let \( \underline{u}=(u_1,u_2,\ldots,u_k) \) be a complete and irredundant list of the complex numbers appearing in \( \uz \) and let \( A_i =  \{\; j \;|\; z_j=u_i \;\} \), the set of indices that record where \( u_i \) appears in \( \uz \). Fix an ordering of the elements of \( A_i \) and let \( \underline{a}^i \) the sequence obtained in this way. For example, if \( \uz = (\alpha,\beta,\alpha,\gamma,\gamma) \) for three distinct \( \alpha,\beta,\gamma \in \CC \), then one possibility is \( u_1=\alpha \), \( u_2=\beta \) and \( u_3 = \gamma \), with \( \underline{a}^1 = (3,1) \), \( \underline{a}^2 = (2) \) and \( \underline{a}^3 = (4,5) \).
\begin{lemma}
  \label{lem:equal-params-gaudin-algebra}
With the above notation, \( \phi_w(\uz;\uq) = \Delta^{\underline{a}^1,\underline{a}^2,\ldots,\underline{a}^k} \circ \phi_w(\underline{u};\uq) \).
\end{lemma}
\begin{proof}
  We check the identity on generators \( g\ox t^s \in U(\hat{\gl}_{r-}) \). First
  \[ \phi_w(\uz;\uq)(g\ox t^s) = \sum_{a=1}^n g^{(a)}(w-z_a)^s + \delta_{s,-1}\chi_{\uq}(g)1^{\otimes n}. \]
  On the other hand,
  \begin{align*}
    \Delta^{\underline{a}^1,\underline{a}^2,\ldots,\underline{a}^k} \circ \phi_w(\underline{u};\uq)(g\ox t^s)
    &= \Delta^{\underline{a}^1,\underline{a}^2,\ldots,\underline{a}^k} \left( \sum_{b=1}^k g^{(b)}(w-u_b)^s + \delta_{m,-1}\chi_{\uq}(g)1^{\otimes k} \right) \\
    &= \sum_{b=1}^k\sum_{a \in A_b} g^{(a)}(w-u_b)^s + \delta_{s,-1}\chi_{\uq}(g)1^{\otimes n}.
  \end{align*}
  Since \( A_1,A_2,\ldots,A_k \) partitions \( \{1,2,\ldots,n\} \) and since \( u_b \) appears in \( \uz \) precisely \( \# A_b \) times, the claim follows.
\end{proof}

\subsection{} \label{defn:iga} There is a commutative subalgebra \( \mathcal{A}\subseteq U(\hat{\gl_r}_-) \), the \emph{universal Gaudin algebra}, that is free commutative on an infinite set of generators. The generators are described in \cite[Corollary 2]{ryb_shift}, and arise from taking repeated derivatives of a set of homogeneous generators for a copy of \( S(\gl_r)^{\gl_r} \) in \( U(\hat{\gl_r}_-) \). We define the {inhomogeneous Gaudin algebra}\adri{There is a bit of inconsistency in the names and naotations.}\noah{I'll change everything to Gaudin, but happy to change it to Bethe if we think that's best} \iain{I like Gaudin: it almost sounds like my name :)} \[ \mathcal{A}(\underline{z}; \underline{q}) := \phi_w(\underline{z}; \underline{q})(\mathcal{A}). \] This algebra is independent of the choice of \( w\in \mathbb{C} \). 

\begin{theorem}[{\cite[Lemma 9.3]{HKRW}}]
  \label{thm:bethe-algebra-commutative}
The inhomogeneous Gaudin algebra \( \mathcal{A}(\underline{z}; \underline{q}) \) is commutative.  Let \( \zz_{\glr}(\underline{q}) \subseteq \glr \) be the centraliser of \(\underline{q} \in \glr \). Then \( \mathcal{A}(\underline{z}; \underline{q}) \) commutes with \( \Delta^{(n-1)}(\zz_{\glr}(\underline{q})) \subseteq U(\glr)^{\otimes n} \). 
\end{theorem}
\subsection{} There is one case that will be of particular interest to us and we will use a special notation. This is when \( n=1 \), so there is one variable \( \underline{z} = (z_1) \). In this case, the algebra \( \mathcal{A}(\underline{z}; \underline{q}) \subseteq U(\glr) \) does not depend on \( z_1 \) and we shall denote the algebra by \( \mathcal{A}(\underline{q}) \).
When $\uq$ is regular this algebra contains the abelian Lie subalgebra \( \{ G_{\underline{h}}: \underline{h} = (h_1, \ldots , h_r)\in \diag(\gl_r) \subseteq \gl_r\} \) where \[ G_{\underline{h}} := \sum_{i<j} \frac{h_i-h_j}{q_i-q_j}
   E_{ij}E_{ji}\in U(\glr), \] see \cite[Proposition 9.5]{HKRW}. It also contains the Cartan subalgebra \( \diag (\glr) \) of \( \glr \), \cite[loc.cit.]{HKRW}.
 \subsection{}
Using the notation from Section~\ref{sec:equal-params-gaudin}, we can describe the algebras \( \uBethe(\uz;\uq) \) when some of the parameters in \( \uz \) coincide.
\begin{proposition}
  \label{prp:equal-params-gaudin-algebra}
For \( \uz \in \CC^n \), \( \uBethe(\uz;\uq) = \uBethe(\underline{u};\uq)^{\underline{a}^1,\underline{a}^2,\ldots,\underline{a}^k} \). In particular, if \( \uz=(z,z,\ldots,z) \) then \( \uBethe(\uz;\uq) = \uBethe({\uq})^{(1,2,\ldots,n)} = \Delta^{n}(\uBethe({\uq})) \) and if \( \uz = (0,0,\ldots,0,z) \) then \( \uBethe(\uz;\uq) = \uBethe(0,z;\uq)^{(1,2,\ldots,n-1)(n)} \).
\end{proposition}
\begin{proof}
  This follows immediately from the definition and Lemma~\ref{lem:equal-params-gaudin-algebra}.
\end{proof}
 
\label{sec:limits-gaudin-algebras}
\subsection{Limits of inhomogeneous Gaudin algebras}
\label{ztoinfty}
If \( \underline{q}\in \text{diag}(\glr) \) is reqular, then by \cite[Theorem 2]{ryb_shift}, we have \[ \lim_{t\rightarrow \infty} \mathcal{A}(t\underline{z}; \underline{q}) = \mathcal{A}(\underline{q})^{\otimes n} \subseteq U(\glr)^{\otimes n}. \]
\subsection{} In fact, we can strengthen the above result slightly to include small deformations of the path \( t\uz \). This will be useful when we consider homotopies of paths.
\begin{proposition}
  \label{prp:ztoinfty}
Let \( \uz(t): [1,\infty) \rightarrow \Creg^n \) be a path such that \( \lim_{t\to \infty} z_{i+1}(t)-z_i(t) = \infty \). Then 
\[ \lim_{t\rightarrow \infty} \mathcal{A}(t\underline{z}; \underline{q}) = \mathcal{A}(\underline{q})^{\otimes n} \subseteq U(\glr)^{\otimes n}. \]
\end{proposition}
\begin{proof}
  The proof of \cite[Theorem 2]{ryb_shift} applies word for word to the above statement.
\end{proof}
\subsection{}\label{sec:JMlimit}
It will also be desirable to understand a particular limit as \( \uz\to \underline{0} \). Consider a path \( \uz(t): (0,1]\rightarrow \Rdel^n \) such that \( \lim_{t\to 0} z_{i}(t)/z_{i+1}(t) = 0 \). Define \( \uJM_n = \lim_{t\to 0}\uBethe(\uz(t);\underline{0}) \) (the notation comes from the fact that on \( (V_{\varpi_1}^r)^{\ox n} \) this algebra coincides with the algebra of Jucys-Murphy operators).
\begin{proposition}
  \label{prp:JMlimit}
The limiting algebra \( \lim_{t\to 0} \Aa(\uz(t);\uq) \) contains and is generated by \( \Delta^n(\Aa(\uq)) \) and \( \uJM_n \).
\end{proposition}
\begin{proof}
  This is a generalisation of \cite[Proposition~10.16 (1)]{HKRW} and the same proof works here. \end{proof} A very similar argument will be used in Lemma~\ref{lem:limit-bethe-algebras} in a more complicated situation.  
\subsection{} 
\label{GTfamily} The family of algebras \( \mathcal{A}( \underline{q}) \) with \( \underline{q}\in \Creg^r \) are constant under the action of \( \CC \rtimes \CC^* \) on \( \Creg^r \) by affine shifts and dilations. We thus have a family of algebras over \( \Creg^r/\CC \rtimes \CC^* = \mathbb{P}^1(\CC)^{r+1}_{\text{reg}}/ PSL(2,\CC) \), where we have added the point at infinity to realise the right hand side of this equality. By \cite[Theorem 2.5]{aguirre_felder_veselov_2011} and \cite[Theorem 10.8]{HKRW} this can be extended to a flat family of maximally commutative subalgebras of \( U(\glr) \) over \( \overline{\mathcal{M}}_{0,r+1} \), the moduli space of \( r+1 \) points on genus \( 0 \) curves. We wish to describe one particular limit point. 
\subsection{} 
We embed \( \gl_{r-1} \subset \glr \) as the set of matrices with zero final row and column. The { Gelfand-Tsetlin subalgebra} \( \uGT_r = \langle ZU(\gl_1), ZU(\gl_2), \ldots,ZU(\glr) \rangle \subset U(\glr) \) is the commutative subalgebra generated by the successive centres of these embedded general linear Lie algebras. 
\begin{proposition}[{\cite[Lemma 4]{ryb_shift}}]
  \label{prp:GT-limit}
Let \( \uq = (q_1,q_2,\ldots,q_r) \in \Creg^r \) and \( \uq(t) =(q_1t^{r-1},q_2t^{r-2},\ldots,q_r) \). Then \(  \lim_{t\rightarrow 0} \mathcal{A}(\underline{q}(t)) = \uGT_r \), the Gelfand-Tsetlin subalgebra. 
\end{proposition}
\section{Statement of Main Theorem }
\label{sec:eigensp-semist-tabl}
In this section we work with \( \underline{z} = (z_1,z_2,\ldots,z_n) \in \Rdel^n \) and \( \underline{q} = (q_1,q_2,\ldots,q_r) \in \Rdel^r \).
\subsection{Notation} We are going to be interested in the spectrum of inhomogeneous Gaudin algebras. Recall from the introduction if \( M \) is a \( \mathcal{A}(\underline{z};\underline{q}) \)-module, we will write \( \mathcal{E}_{\underline{z};\underline{q}}(M) \) instead of \( \mathcal{E}_{\mathcal{A}(\underline{z};\underline{q})}(M) \). When \( n=1 \) we will denote the spectrum of \( \Aa(\uq) \) on \( M \) by \( \Ee_{\uq}(M) \). For a weight \( \lambda \), let \( V^r_\lambda \) be the corresponding irreducible \( \glr \)-module. Usually we will drop the superscript and write \( V_\lambda \), however occasionally we will include it for clarity. If \( M = V_{\lambda^{(1)}} \ox V_{\lambda^{(2)}} \ox \cdots \ox V_{\lambda^{(n)}} \) for some tuple of weights \( \underline{\lambda} = (\lambda^{(1)},\lambda^{(2)},\ldots,\lambda^{(n)}) \) then we will simplify notation further by writing  \( \mathcal{E}_{\underline{z};\underline{q}}(\underline{\lambda}) \) and  finally, if \( \underline{\lambda} = (k_1\varpi_1,k_2\varpi_1,\ldots,k_n\varpi_1) \) we will write \( \mathcal{E}_{\underline{z};\underline{q}}(\underline{k}) \). 
\subsection{Combinatorial Description} The Lie algebra \( \gl_r \) acts by left multiplication on \( \CC[\Mat_{r\times n}] \), polynomial functions on the space of complex matrices. The spectrum of the inhomogeneous Gaudin algebras on this space is the main object of our story. We can give an alternative description as follows. The algebra \( \glr \) acts on \( \CC[x_1,x_2,\ldots,x_r] \)  by differential operators where \( E_{ij} \) acts by \( x_i\partial_j \). The submodule of homogeneous degree \( k \) polynomials is isomorphic to \( V_{k\varpi_1} \). There is an isomorphism
\[ \CC[\Mat_{r\times n}] \longrightarrow \CC[x_1,x_2,\ldots,x_r]^{\ox n}; x_{ij} \mapsto x_i^{(j)}. \]
\subsection{}\label{wt-sp-labels} A weight basis for \( \CC[\Mat_{r\times n}] \) is given by the monomials \( x^A = \prod_{i,j} x_{ij}^{A_{ij}} \), where \( A = (A_{ij}) \in \Mat_{r\times n}(\NN) \). In other words, this weight basis is labelled by \( \Mat_{r\times n}(\NN) \). We now construct a bijection $$ \alpha: \Ee_{\uz;\uq}(r\times n) :=  \Ee_{\uz;\uq}(\CC[\Mat_{r\times n}]) \longrightarrow  \Mat_{r\times n}(\NN). $$ 
\subsection{}
The weight spaces of \( \CC[x_1,x_2,\ldots,x_r] \) are spanned by the monomials and so one dimensional. Since \( \Aa(\uq) \) contains the Cartan subalgebra, it thus acts with simple spectrum. It follows that \( \Aa(\uq)^{\ox n} \) has simple spectrum on \( \CC[\Mat_{r\times n}] \). Recall from \ref{ztoinfty} that \( \lim_{t\rightarrow \infty} \mathcal{A}(t\underline{z}; \underline{q}) = \mathcal{A}(\underline{q})^{\otimes n} \). We can thus conclude that \( \Aa(\uz;\uq) \) has simple spectrum on \( \CC[\Mat_{r\times n}] \) for generic \( \uz \) (alternatively this can be deduced from \cite[Corollary 6 and proof of Corollary 4]{Feigin:2010df}). Following \ref{continuation}, parallel transport along the line \( \{ t \uz: 1\leq t\leq \infty\} \) therefore gives a bijection \[  p_\infty: \mathcal{E}_{\underline{z};\underline{q}}(r\times n)\longrightarrow \mathcal{E}_{\underline{\infty};\uq}(r\times n). \]
 \label{sec:ebasis} Since the weight spaces in \( \CC[\Mat_{r\times n}] \) are one dimensional, they coincide with the spectrum of \( \Aa(\uq)^{\ox n} \), that is \( \Ee_{\underline{\infty};\uq}(r\times n) = \{\CC x^A \mid A \in \Mat_{r\times n}(\NN)\} \). We thus have a bijection
\[  \comb_{\underline{\infty}; \uq}^{r \times n} : \mathcal{E}_{\underline{\infty};\uq}(r \times n)  \longrightarrow \Mat_{r \times n}(\NN); \,\,  \CC x_A \mapsto A.  \] 
After composition with \( p_\infty \), this produces a labelling of the spectrum \( \Ee_{\uz;\uq}(r\times n) \) by a non negative integer matrices:
\[ \alpha = \comb_{\underline{\infty};\uq}^{r \times n} \circ p_{\infty}: \Ee_{\uz;\uq}(r \times n) \longrightarrow \Mat_{r \times n}(\NN). \]
\subsection{} \label{ztozero} We now construct a mapping $$\beta: \mathcal{E}_{\uz, \uq}(r\times n) \longrightarrow \bigsqcup_{\lambda \in \Part_{\le r}} \SSYT(\lambda)$$
\subsection{} Choose a path \( \underline{z}(t) \in \Rdel^n \) for \( t \in (0,1] \) such that \( \underline{z}(1)=\underline{z} \) and  \( \lim_{t \to 0} z_i(t) =0 \). By construction, \( \mathcal{A}(\underline{z}; \underline{q}) \) is the image of the algebra \( \mathcal{A} \) under the mapping \[ \phi_w(\underline{z}; \underline{q}) = (\phi_{w,\underline{z}} \otimes (\chi_{\uq}\circ \phi_{\infty}))\circ \Delta: U(\hat{\glr}_-) \rightarrow U(\glr)^{\otimes n}. \]
Therefore,  \( \lim_{t\to 0}\mathcal{A}(\underline{\underline{z}(t)}; \underline{q}) \) contains the image of \( \mathcal{A} \) under \( (\phi_{w,\underline{0}} \otimes (\chi_{\uq}\circ \phi_{\infty}))\circ \Delta \), which is \( \Delta^{n}\mathcal{A}(\uq) \) (see~\ref{lem:equal-params-gaudin-algebra}). Thus, by parallel transport along \( \uz(t) \) (and restriction to \( \Delta^n\Aa(\uq) \)) we obtain a map
\[ p_{\uz=0}: \Ee_{\uz;\uq}(r \times n) \longrightarrow \Ee_{\uq}(r \times n) \]
where \( \Ee_{\uq}(r \times n) = \Ee_{\uq}(\CC[\Mat_{r \times n}]) \), the spectrum of \( \Aa(\uq) \) acting diagonally on \( \CC[\Mat_{r \times n}] \).  
 Since \( \Aa(\uq) \) contains \( ZU(\glr) \) and the finite dimensional irreducible \( \glr \)-modules are determined by their central character, an eigenspace \( E \in \Ee_{\uq}(r \times n) \) is contained in an isotypic component of \( \CC[\Mat_{r \times n}] \). This produces a map \( \xi:  \Ee_{\uq}(r \times n) \longrightarrow \bigsqcup_{\lambda} \Ee_{\uq}(\lambda) \).
For any \( \underline{q}\in \Rdel^r \), Proposition~\ref{prp:GT-limit} and parallel transport along the path \( \uq(t) = (t^{r-1}q_1,t^{r-2}q_2,\ldots,q_r) \) gives an identification \( \kappa: \Ee_{\uq}(\lambda) \longrightarrow \Ee_{\underline{0}}(\lambda) \).
 \label{sec:GTdeg} Now denote the spectrum of the Gelfand-Tsetlin algebra \( \uGT_r \), acting on \( V_\lambda \) by \( \Ee_{\underline{0}}(\lambda) \). The set \( \Ee_{\underline{0}}(\lambda) \) is identified with the \( \SSYT_r(\lambda) \) as follows.
Let \( E \in \mathcal{E}_{\underline{0}}(\lambda) \). For each \( 1 \le i \le r \) $E$ is contained in an irreducible \( \gl_i \)-submodule \( V^i_{\lambda^{(i)}} \subset V_{\lambda} \). Thus to \( E \) we associate a sequence \( \lambda^{(1)},\lambda^{(2)},\ldots,\lambda^{(r)}=\lambda \) of partitions, where \( \lambda^{(i-1)} \subset \lambda^{(i)} \) and the skew-shape \( \lambda^{(i)}\setminus\lambda^{(i-1)} \) does not have more than one box in any column. From such a sequence, we produce a tableau by filling the boxes corresponding to \( \lambda^{(i)} \setminus \lambda^{(i-1)} \) with \( i \). The map \[  \comb_{\underline{0}}: \Ee_{ \underline{0}}(\lambda) \longrightarrow \SSYT_r(\lambda)  \] is a bijection. We will write \( \comb_{\underline{0}}^r \) for this map if we need to emphasise that it is induced by \( \uGT_r \subset U(\glr) \).
Together these mappings produce $$\beta = \comb_{\underline{0}}\circ \kappa \circ \xi \circ p_{\uz = 0} : \mathcal{E}_{\uz, \uq}(r\times n) \longrightarrow  \bigsqcup_{\lambda \in \Part_{\le r}} \SSYT(\lambda) $$
\subsection{}
\label{eq:Smap}
Combining the previous sections we obtain a mapping
\begin{equation*}
  S = \beta \circ \alpha^{-1} : \Mat_{r \times n}(\NN) \longrightarrow \bigsqcup_{\lambda \in \Part_{\le r}} \SSYT(\lambda).
\end{equation*}
\subsection{Duality}\label{duality} 
The algebra \( \gl_n \) acts on \( \CC[\Mat_{r \times n}] \) by transposed matrix multiplication. We can consider the \( \gl_n \) module \( \CC[y_1,\ldots,y_n] \) where \( E_{ij} \) acts by \( y_i\partial_j \) and we obtain the following identifications
\[ \CC[x_1,\ldots,x_r]^{\ox n} \cong \CC[\Mat_{r\times n}] \cong \CC[y_1,\ldots,y_n]^{\ox r}  \]
all of which can be considered \( \glr \oplus \gl_n \) modules (the actions commute). In addition, there are actions of \( U(\glr)^{\ox n} \) and \( U(\gl_n)^{\ox r} \) (these actions do not commute). We will denote the relevant maps by
\begin{align*}
  \pi^r &: U(\gl_r)^{\ox n} \longrightarrow \End(\CC[\Mat_{r\times n}]) \\
  \pi^n &: U(\gl_n)^{\ox r} \longrightarrow \End(\CC[\Mat_{r\times n}]).
\end{align*}
\subsection{}
We now define a map analogous to~\ref{eq:Smap} by sending \( \uq \to 0 \). Choose a path \( \underline{q}(t) \in \Rdel^r \) for \( t \in (0,1] \) such that \( \uq(1)=\uq \) and  \( \lim_{t \to 0} q_i(t) =0 \). In this limit \( \lim_{t\to 0}\uBethe(\underline{z}; \underline{q}(t)) \) contains  \( \uBethe(\underline{z}; \underline{0}) \subset U(\glr)^{\ox n} \). According to \cite[Theorem~6]{Mukhin:2009in} we have that \( \pi^r(\uBethe(\underline{z};\underline{0})) = \pi^n(\Delta^r\uBethe(\underline{z})) \). This allows us to repeat the construction from sections~\ref{ztozero} to get a map $\beta'$, obtained by composing \( \Ee_{\uz;\uq}(r\times n)\longrightarrow \Ee_{\uz}(r \times n) \) with the operation to then produce a semistandard tableau. Putting this all together gives us a map
\begin{equation}
  \label{eq:Tmap}
  T = \beta' \circ \alpha^{-1}: \Mat_{r \times n}(\NN) \longrightarrow \bigsqcup_{\lambda \in \Part_{\le n}} \SSYT(\lambda).
\end{equation}
\subsection{Main Result} We can now state the main result of the paper.
\begin{theorem}\label{thm:main}
The assignment \( A \mapsto (S(A),T(A)) \) agrees with the RSK correspondence. More precisely, \( S(A)=P(A) \) and \( T(A) = Q(A) \) for all \( A \in \Mat_{r\times n}(\NN) \).
\end{theorem}
\section{Crystal structures on \( \Ee_{\uq}(\lambda) \)}
\label{sec:crystal-structures}
\subsection{Crystal structures from $ \mathcal{A}(\underline{q})$}
 \label{prp:GT-HKRW-crystals-agree}
 In order to prove Theorem \ref{thm:main} we will first understand better the mapping of \ref{sec:GTdeg} \[ \comb_{\underline{0}}\circ \kappa : \mathcal{E}_{\underline{q}}(\lambda) \longrightarrow \SSYT_r(\lambda). \] 
We know that \( \SSYT_r(\lambda) \) is an irreducible \( \glr \)-crystal. Thanks to ~\cite[Proposition 12.2]{HKRW}, there is a crystal structure defined directly on \( \mathcal{E}_{\underline{q}}(\lambda) \) through the representation theory of \( \mathcal{A}(\underline{q}) \) on \( V_{\lambda} \). 
\begin{proposition}
The mapping \( \comb_{\underline{0}}\circ \kappa  \) is an isomorphism of \( \glr \)-crystals.
\end{proposition}
\begin{proof}
We will proceed by induction on \( r \), noting that for \( r=2 \) the result is immediate since the weight spaces of \( V_\lambda \) are all one-dimensional.
 
Recall from \ref{sec:GTdeg} the mapping 
that \( \comb_{\underline{0}}\circ \kappa  \) is defined using parallel transport along the path \( \uq(t) = (t^{r-1}q_1,t^{r-2}q_2,\ldots, q_r)\in \Rdel^r \) for \( t \in (0,1] \).  Recall from Section~\ref{GTfamily}, the mapping from \( \Creg^r \) to the space of commutative subalgebras of \( U(\glr) \) extends to a mapping from \( \overline{\mathcal{M}}_{0,r+1} \). For \( x\in \overline{\mathcal{M}}_{0,r+1} \) we denote the corresponding algebra by \( \mathcal{A}(x) \). With this notation, \( \mathcal{A}(\underline{q}) \) is the algebra corresponding to \( x = (q_1, q_2, \ldots , q_r, \infty) \in \overline{\mathcal{M}}_{0,r+1} \). We now consider the path \( \uq(t) \) in this moduli space, and thus include the limit point \(\uq(0) \). 
We will label the marked points in \( \overline{\mathcal{M}}_{0,r+1} \) by \( 1,2,\ldots,r,\infty \). Let \( \overline{\mathcal{M}}(i,j) \) be the codimension one boundary component of \( \overline{\mathcal{M}}_{0,r+1} \) where precisely the marked points \( i, i+1, \ldots , j \) lie in the same irreducible component. 
 We are going to consider a path \( \tilde{\uq}(t) \) in \( \overline{\mathcal{M}}_{0,r+1} \) that is homotopy equivalent to \( \uq(t) \). It is the composition of two paths, \( \tilde{\uq}^1(t) \) and \( \tilde{\uq}^2(t) \), which are defined as follows. 
 \begin{itemize}
     \item
We let \( \tilde{\uq}^1(t) = (tq_1,tq_2,\ldots,tq_{r-1},q_r) \) for \( t\in (0,1] \). The limit point \( \tilde{\uq}^1(0) \) is the stable curve in \( \overline{\mathcal{M}}(r,\infty) \) where one component has marked points at \( q_r \) and \( \infty \) and a node at \( 0 \), while the second component has marked points at \( q_1,q_2,\ldots,q_{r-1} \) and a node at \( \infty \). 
\item We let \( \tilde{\uq}^2(t)\in \overline{\mathcal{M}}(r,\infty) \) be the stable curve where the first component has marked points at \( q_r \) and \( \infty \) and a node at \( 0 \), and the other component(s) has marked points at \( t^{r-2}q_1,t^{r-3}q_2,\ldots,q_{r-1} \) and a node at \( \infty \). 
 \end{itemize}
We will first show that parallel transport along \( \tilde{\uq}^1(t) \) and \( \tilde{\uq}^2(t) \) both induce \( \gl_{r-1} \)-crystal morphisms.
First we deal with \( \tilde{\uq}^2(t) \). Let \( \iota: U(\gl_{r-1}) \longrightarrow U(\glr) \) be the embedding induced by \( \gl_{r-1} \subset \glr \) as the matrices with zero final row and column. By \cite[Corollary~10.12]{HKRW} the algebra, \( \mathcal{A}(\tilde{\uq}^1(0)) \) at the initial point is generated by \( \iota\mathcal{A}((q_1,q_2,\ldots,q_{r-1})) \) and \( \mathcal{A}((0,\ldots,0,q_r)) \subset U(\glr)^{\gl_{r-1}} \). Since the decomposition of  \( V_{\lambda} \) into \( \gl_{r-1} \) representations is multiplicity free, we have  \[  \mathcal{E}_{\tilde{\uq}^1(0)}(\lambda) = \bigsqcup_{\mu} \mathcal{E}_{(q_1,q_2,\ldots,q_{r-1})}(\mu),  \] where the union ranges over \( \mu \subset \lambda \) such that \( \lambda \setminus \mu \) has at most one box in every column, see \cite[Corollary 10.13]{HKRW}. The induction assumption is that the mapping \( \comb^{r-1}_{(q_1, \ldots , q_{r-1})}(\mu) \) is a \( \gl_{r-1} \)-crystal isomorphism for every \( \mu \). These are precisely the maps induced by parallel transport along \( \tilde{\uq}^2(t) \).
    
  Now we consider \( \tilde{\uq}^1(t) \). The crystal operator \( \tilde{e}_i \) in the crystal structure on \( \mathcal{E}_{\underline{q}}(\lambda) \) is defined using parallel transport to any point on \( \overline{\mathcal{M}}(i,i+1) \). When \( 1 \le i \le r-2 \), we have that \( \overline{\mathcal{M}}(i,i+1) \cap \overline{\mathcal{M}}(r,\infty) \) is nonempty. By the operadic nature of \( \overline{\mathcal{M}}_{0,r+1} \), we therefore see that parallel transport along \( \tilde{\uq}^1(t) \) is a morphism of \( \gl_{r-1} \)-crystals since the parallel transport defining \( \tilde{e}_i \) factors through the parallel transport to \( \tilde{\uq}^1(0) \). 
  
  It follows from these two paragraphs that that parallel transport along the composition of \( \tilde{\uq}^1(t) \) and \( \tilde{\uq}^2(t) \) is a \( \gl_{r-1} \)-crystal morphism. Since this composition is homotopic to \( \uq(t) \), this shows \( \comb^r_{\underline{q}}(\lambda) \) is a mapping of \( \gl_{r-1} \)-crystals. This implies that this mapping sends \( \gl_{r-1} \) highest weights to elements of \( \SSYT_{r}(\lambda) \) that are also \( \gl_{r-1} \) highest weights. But such semistandard tableaux are uniquely determined by their weights as a \( \glr \)-representation. Since \( \comb_{\underline{0}}\circ \kappa   \) preserves weight spaces by construction, it follows that it is a \( \glr \)-crystal mapping. 
  \end{proof}

\section{Moduli of inhomogenous Gaudin algebras}
\label{sec:moduli-bethe}
For \( (\uz,\uq) \in \Creg^n \times \Creg^r \), the algebras \( \uBethe(\uz;\uq) \subseteq U(\glr)^{\otimes n} \) have constant Hilbert series with respect to the PBW filtration (see \cite[Section~9.4]{HKRW}). Thus if we fix \( \uq \in \Creg^r \), then \( \Creg^n \) maps into the space of subalgebras of \( U(\glr)^{\otimes n} \) with the same Hilbert series as \( \uBethe(\uz;\uq) \). Denote the closure of the image of \( \Creg^n \) by \( \mathfrak{X}_{\uq} \), the compactified moduli space of inhomogeneous Gaudin algebras. If \( x \in \mathfrak{X}_{\uq} \), we denote the corresponding algebra by \( \Aa(x;\uq) \). The purpose of this section is to present a homotopy of paths in the space of \( \mathfrak{X}_{\uq} \). This will be the main technical ingredient in the proof of Theorem~\ref{thm:main}.
\subsection{}
\label{sec:collisionpathsdefn} We call a path \( \uz (t): \RR_{>0} \longrightarrow \Creg^n \) a \emph{collision path} if it satisfies the following four properties:
\begin{itemize}
\item \emph{ordered}: \( \uz (t) \in \Rdel^n \) for all \( t \in \RR_{>0} \),
\item \emph{monotonicity}: the functions \( z_i(t) \) are monotonic for $1\leq i \leq n$,
\item \emph{limiting behaviour}: \( \lim_{t \to 0} z_i(t) = 0  \) and \( \lim_{t \to \infty} z_i(t) = \infty  \) for all \( 1\le i \le n \), 
\item \emph{asymptotic ordering}: \( \lim_{t \to 0} z_{i+1}(t)/z_{i}(t) = \lim_{t \to \infty} z_{i+1}(t) - z_i(t) = \infty \) for all \( 1\le i < n \).
\end{itemize}
These conditions ensure that
\begin{itemize}
\item \( \lim_{t\to \infty} \uBethe(\uz(t);\uq) = \Aa(\uq)^{\ox n} \) (see Section~\ref{ztoinfty}), 
\item \( \lim_{t\to 0} \uBethe(\uz(t);\uq) \) is the algebra generated by \( \Delta^n\uBethe(\uq) \) and \( \uJM_n \) (see Section~\ref{sec:JMlimit}).
\end{itemize}
\begin{example}
  \label{exm:path-for-monodromy}
An example of a collision path is \( z_i(t) = t^{n-i+1}(1+t^2)^{i-1} \). If we want this to pass through a specific point \( (z_1,z_2,\ldots,z_n) \), at say \( t=1 \) for example, we can take the path \( z_i(t) = 2^{1-i}z_it^{n-i+1}(1+t^2)^{i-1} \).
\end{example}
Fix a choice of collision path. Let \( \RR_{\ge 0}^\infty =\RR_{\ge 0} \cup \{\infty\} \). We can now define a path \( \gamma:\RR_{\ge 0}^\infty \longrightarrow \mathfrak{X}_{\uq} \) given by \( \gamma(t)=\uBethe(\uz(t);\uq) \) when \( t \neq 0,\infty \), and given by the limiting algebras at the start and end points. This is the path whose parallel transport will be studied in Theorem~\ref{thm:HKRW-to-RSK}.
\subsection{Remark}
  \label{rem:framed-stable-curves}
  When \( \underline{q} \) is regular, the moduli space \( \mathfrak{X}_{\uq} \) is expected to be isomorphic to the moduli space of framed, stable, genus zero curves with \( n+1 \) marked points (see~\cite[Remark~10.20]{HKRW}). These are chains of genus zero curves, with points marked by \( 1,2,\ldots,n \) and \( \infty \), and a choice of nonzero tangent vector at \( \infty \).
  Suppose that \( \mathfrak{X}_{\uq} \) is indeed isomorphic to this moduli space, then the path \( \gamma \) corresponds to a path starting at the curve \( \gamma(0) \) with \( n-1 \) components arranged linearly, the labels \( 1,2 \) on the first component, \( i+1 \) on the \( i^{\text{th}} \) component, and \( \infty \) on the final component. For \( 0 < t < \infty \), the curve \( \gamma(t) \) is a single component with labels at the pointed determined by \( z(t) \). For \( t=\infty \), the curve \( \gamma(\infty) \) is again \( n \) components, this time, all sharing a single nodal point marked \( \infty \). 
  Now we would like to apply induction on \( n \) to this path later. The moduli space of curves interpretation suggests how to do this. We use a different path \( \gamma_0 \). We start at the same stable curve \( \gamma_0(0)=\gamma(0) \), but now move only the point marked by \( n \). That is we move through stable curves \( \gamma_0(t) \) with \( n-1 \) components, arranged linearly, the first marked by \( 1,2 \) and the \( i^{\text{th}} \) marked by \( i+1 \), but now the last component marked by \( n \) and \( \infty \).
  We ask that the \( n^{\text{th}} \) marked point collides with \( \infty \) when \( t=1 \), that means, \( \gamma_0(1) \) is the same stable curve as \( \gamma_0(0) \) except that \( \infty \) labels the node between the final two components, and \( n \) labels a point on the final component. For \( 1<t<\infty \) we then inductively use a collision path to move the marked points \( 1,2,\ldots,n-1 \) to \( \infty \) to arrive at \( \gamma_0(\infty)=\gamma(\infty) \).
  In the moduli space of curves, it is clear these two paths are homotopy equivalent. However since we do not know that \( \mathfrak{X}_{\uq} \) is indeed isomorphic to this space, we cannot use this fact. The following proves that the corresponding paths in \( \mathfrak{X}_{\uq} \) are homotopy equivalent directly. 
\subsection{} The idea of a collision path is that it parametrises \( n \) particles that collide at zero and infinity in a specified order, simultaneously. Our aim now, is to present a new path in \( \mathfrak{X}_{\uq} \), where the \( n^{\text{th}} \) particle is sent from zero to infinity first, and then the remaining particles follow. This will allow us to use induction when calculating the parallel transport along a collision path. We show below that this new path is homotopy equivalent to \( \gamma \). So the strategy can be summarised as writing down what would be an explicit homotopy in the moduli space of curves and showing explicitly that it is in fact a homotopy in \( \mathfrak{X}_{\uq} \).
                     
\subsection{}
Define a deformation \( \uz:(0,1] \times \RR_{>0} \longrightarrow \Creg^n \) of our fixed collision path. For \( 1 \le i \le n-1 \) let
\begin{align*}
  z_i(s,t) &=
  \begin{cases}
    z_{i}(st) &\text{when } 0 \le t \le 1, \\
    z_{i}(t-1+s) &\text{when } t > 1,
  \end{cases}
\intertext{and let}                  
z_n(s,t) &=
  \begin{cases}
    \frac{2-t}{1-t+s}z_{n}(t) &\text{when } 0 \le t \le 1, \\
    s^{-1}z_{n}(t) &\text{when } t > 1,
  \end{cases}
\end{align*}
\begin{lemma}
  \label{lem:collision-path-deformation}
  Fix \( s \in (0,1] \). Then \( (z_i(s,\cdot))_i \) is a collision path.
\end{lemma}
\begin{proof}
  We first observe that the path is continuous. Now we check the four properties defining collision paths.
  
  To see the ordering property, clearly \( z_i(s,t) < z_{i+1}(s,t) \) for any fixed \( s \) and \( t \) and \( 1 \le i \le n-2 \). What is left to observe is that \( z_{n-1}(s,t) < z_{n}(s,t) \). When \( t \le 1 \), we have
  \[ z_{n-1}(s,t) = z_{n-1}(st) \le z_{n-1}(t) < z_n(t) \le \frac{2-t}{1-t+s}z_{n}(t) = z_n(s,t) \]
  where the first inequality follows by monotonicity, the second by the ordering, and the third since \( \frac{2-t}{1-t+s} \geq 1 \) when \( s \in (0,1] \) and \( t\le 1 \). Now when \( t > 1 \), we have
  \[ z_{n-1}(s,t) = z_{n-1}(t-1+s) \le z_{n-1}(t) < z_n(t) \le s^{-1}z_{n}(t) = z_n(s,t) \]
  for similar reasons.
  Monotonicity follows from the observation that if \( t < 1 < t' \) then \( st < t'-1 +s \). The limiting behaviour and asymptotic ordering both follow from the corresponding properties of \( \uz(t) \).
\end{proof}
\subsection{} Now we can define the path \( \gamma_0:\RR_{\ge 0}^\infty \longrightarrow \mathfrak{X}_{\uq} \) given by \( \gamma_0(t) = \lim_{s\to 0} \uBethe_q(z(s,t)) \) when \( t \neq 0,\infty \) and defined by the appropriate limiting algebras at these endpoints. While it is immediately clear that \( \uz(s,t) \) is a homotopic to \( \uz(t) \) for any fixed \( s \in (0,1] \), the same cannot be said for \( \gamma(t) \) and \( \gamma_0(t) \). To see this, we must show that \( \gamma_0(t) \) is continuous and to do this we will explicitly calculate the algebras corresponding to each point. 
\begin{lemma}
  \label{lem:limit-bethe-algebras}
  The algebra \( \lim_{s\to 0} \uBethe(\uz(s,t);\uq) \)
  \begin{enumerate}
  \item is equal to \( \uBethe(z_1(t-1),\ldots,z_{n-1}(t-1);\uq)\otimes \uBethe(\uq) \) when \( t > 1 \), and
  \item is generated by \( \uJM_{n-1} \) and \( \uBethe\left( 0,\frac{2-t}{1-t}z_n(t);\uq \right)^{(1,2,\ldots,n-1)(n)} \) when \( 0 < t < 1 \), and
  \item is generated by \( \uJM_{n-1}\otimes \Aa(\uq) \) and \( \Delta^{n-1}\Aa(\uq) \) when \( t=1 \).
  \end{enumerate}
\end{lemma}
\begin{proof}
  We imitate the method of proof used in~\cite[Section~5]{ryb_shift} and~\cite[Section~10]{HKRW}. In particular, the algebra \( \uBethe(\uz;\uq) \) is generated by the coefficients of the principal parts of the Laurent series of \( f_l(w;\uz;\uq) = \phi_w(\uz;\uq)(S_l) \) where \( S_l \in \uBethe \) are the free generators described in~\cite{Feigin:1994tc}. We also note that \( \lim_{w\to\infty} \phi_w =\varepsilon \), the counit map (see~\cite[Lemma~2]{ryb_shift}).
  First we deal with \( t > 1 \). The generators of \( \uBethe(\uz(s,t);\uq) \) are the coefficients of the principal part of the Laurent series of
  \begin{align*}
    f_l(w;\uz(s,t);\uq) &= f_l(w;z_1(t-1+s),z_2(t-1+s),\ldots,z_{n-1}(t-1+s),s^{-1}z_n(t);\uq),
  \end{align*}
  so we have
  \begin{align*}
    \lim_{s \to 0} f_l(w;\uz(s,t);\uq) &= \lim_{s \to 0} \phi_{w-z_1(t-1+s)} \ox \cdots \ox \phi_{w-z_{n-1}(t-1+s)} \ox \phi_{w-s^{-1}z_n(t)} \ox (\chi_{\uq} \circ \phi_\infty) \circ \Delta^{(n+1)}(S_l) \\
                                     &= \phi_{w-z_1(t-1)} \ox \cdots \ox \phi_{w-z_{n-1}(t-1)} \ox \varepsilon \ox (\chi_{\uq} \circ \phi_\infty) \circ \Delta^{(n+1)}(S_l) \\
    &= f_l(w;z_1(t-1),\ldots,z_{n-1}(t-1);\uq)^{(1,2,\ldots,n-1)}. 
  \end{align*}
  Thus \( \uBethe(z_1(t-1),\ldots,z_{n-1}(t-1);\uq)^{(1,2,\ldots,n-1)} \subseteq \lim_{s\to0}\uBethe(\uz(s,t);\uq) \). Now consider the principal part of the Laurent series of \( f_l(w;\uz(s,t);\uq) \) at \( w=z_n(s,t)=s^{-1}z_n(t) \). This is the same as the principal part of the Laurent series of \( f_l(w+s^{-1}z_n(t);\uz(s,t);\uq) \) at \( w=0 \) and we have
  \begin{align*}
    \lim_{s\to 0} f_l(w+ & s^{-1}z_n(t);\uz(s,t);\uq) \\ &= \phi_{w+s^{-1}z_n(t)-z_1(t-1+s)} \ox \cdots \ox \phi_{w+s^{-1}z_n(t)-z_{n-1}(t-1+s)} \ox \phi_{w} \ox (\chi_{\uq} \circ \phi_\infty) \circ \Delta^{(n+1)}(S_l) \\
                                                 &= \varepsilon^{\ox n-1} \ox \phi_w \ox (\chi_{\uq} \circ \phi_\infty) \circ \Delta^{(n+1)}(S_l) \\
    &= f_l(w;0;{\uq})^{(n)}
  \end{align*}
  and so \( \uBethe(\uq)^{(n)} \subseteq \lim_{s\to0}\uBethe(\uz(s,t);\uq) \).
  To see that \( \uBethe(z_1(t-1),\ldots,z_{n-1}(t-1);\uq)^{(1,2,\ldots,n-1)}\ox \uBethe_q = \lim_{s\to0}\uBethe(\uz(s,t);\uq) \), we use~\cite[Proposition~9.10]{HKRW} to see that both are free polynomial algebras of the same transcendence degree.
  Now consider the case when \( 0 < t \le 1 \). First of all, by taking the limit and using Proposition~\ref{prp:equal-params-gaudin-algebra}, we see that
  \[ \uBethe\left( 0,\frac{2-t}{1-t}z_n(t);\uq \right)^{(1,2,\ldots,n-1)(n)} = \uBethe\left(0,0,\ldots,0,\frac{2-t}{1-t}z_n(t);\uq \right) \subseteq \lim_{s \to 0} \uBethe(\uz(s,t);\uq). \]
  According to \cite[Lemma~9.2]{HKRW}, \( \uBethe(\uz(s,t);\uq) = \uBethe(z_{n-1}(st)^{-1}\uz(s,t);z_{n-1}(st)\uq) \) and so the limiting algebra contains the coefficients of the principal parts of the Laurent series of
  \begin{align*}
    \lim_{s\to0} f_l(&w;z_{n-1}(st)^{-1}\uz(s,t);z_{n-1}(st)\uq)
    = \lim_{s\to\infty} f_l\left( w; \frac{z_1(st)}{z_{n-1}(st)}, \ldots, \frac{z_{n-2}(st)}{z_{n-1}(st)},1,\frac{z_n(st)}{z_{n-1}(st)}; z_{n-1}(st)\uq \right) \\
    &= \lim_{s\to\infty} \phi_{w-\frac{z_1(st)}{z_{n-1}(st)}} \otimes \cdots \otimes \phi_{w-\frac{z_{n-2}(st)}{z_{n-1}(st)}} \ox \phi_{w-1} \otimes \phi_{w-\frac{z_n(st)}{z_{n-1}(st)}} \ox (\chi_{z_{n-1}(st)\uq} \circ \phi_{\infty}) \circ \Delta^{n+1}(S_l).
  \end{align*}
  By the fact that \( \uz(t) \) is a collision path,
  \[ \lim_{s\to0} \frac{z_i(st)}{z_{n-1}(st)} = 0 \text{ for } 1 \le i < n-1 \text{ and } \lim_{s\to0} \frac{z_n(st)}{z_{n-1}(st)} = \infty \]
  So \( \lim_{s\to0} \phi_{w-\frac{z_n(st)}{z_{n-1}(st)}} = \varepsilon \). By the cocommutativity of \( \Delta \) we get,
  \[ \lim_{s\to0} f_l(w;z_{n-1}(st)^{-1}\uz(s,t);z_{n-1}(st)\uq) = \lim_{s\to0} f_l(w;z_1(s,t),\ldots,z_{n-1}(s,t);0)^{(1,2,\ldots,n-1)} \]
  and we thus have that \( \uJM_{n-1} \subseteq \lim_{s\to0} \uBethe(\uz(s,t);\uq) \).
  To see that the limiting algebra is generated by these two subalgebras note first by a theorem of Knop~\cite{Knop:1994ei} the centraliser of \( (U(\glr)^{\ox 2})^{\glr\oplus \glr} \) embedded via \( \Delta^{(1\ldots n-1)(n)} \) in \( U(\glr^{\oplus n}) = U(\glr)^{\ox n} \) is isomorphic to
  \[ (U(\glr)^{\ox n})^{\Delta^{(1\ldots n-1)(n)}(\glr\oplus\glr)} \ox_{(U(\glr)^{\ox 2})^{\glr\oplus \glr}} \Delta^{(1\ldots n-1)(n)}(U(\glr)^{\ox 2}). \]
  The algebra \( \uJM_{n-1} \) is contained in the first tensor factor and \( \uBethe\left( 0,\frac{2-t}{1-t}z_n(t);\uq \right)^{(1,2,\ldots,n-1)(n)} \) is contained in the second. Thus
  \[ \lim_{s \to 0} \uBethe(\uz(s,t);\uq) \subseteq \uJM_{n-1} \ox_{(U(\glr)^{\ox 2})^{\glr\oplus \glr}} \uBethe\left( 0,\frac{2-t}{1-t}z_n(t);\uq \right)^{(1,2,\ldots,n-1)(n)}. \]
  The tensor factors on the right are free polynomial algebras so to check equality we will show that the right hand side has the same number of algebraically independent generators of the same degrees as \( \uBethe(\uz(s,t);\uq) \). This can be seen by~\cite[Section~9.4 and Proposition~9.10]{HKRW} which implies there is a degree preserving bijection between a set of algebraically independent generators of \( \uJM_{n-1} \) and the principal parts of the Laurent expansions of \( f_l(w;\uz(s,t);q) \) at \( w=z_1,z_2,\ldots,z_{n-1} \), and similarly between generators of \( \uBethe\left( 0,\frac{2-t}{1-t}z_n(t);\uq \right)^{(1,2,\ldots,n-1)(n)} \) and the principal parts of the Laurent expansions of \( f_l(w;\uz(s,t);q) \) at \( w=\infty \). Together, these principal parts form a set of algebraically independent generators of \( \uBethe(\uz(s,t);\uq) \). Since the limit preserves the number and degree of generators (i.e. the Hilbert polynomial) this shows that \( \lim_{s \to 0}\uBethe(\uz(s,t);\uq) \) is generated by the desired algebras.
  Finally, we consider the case \( t=1 \). We have \( \uz(s,1) = (z_1(s),z_2(s),\ldots,z_{n-1}(s),s^{-1}z_n(1)) \) and so the same analysis of the principle parts of \( f_l(w;z_{n-1}(s)^{-1}\uz(s,1);\uq) \) as above will show that \( \uJM_{n-1} \subset \lim_{s\to 0} \Aa(\uz(s,1);\uq) \). Similarly, considering the principle parts of \( f_l(w+s^{-1}z_n(1); \uz(s,1);\uq) \) at \( w=0 \) will show that \( (\Aa(\uq)^{\otimes 2})^{(1,2,\ldots,n-1)(n)} \subset \lim_{s\to 0} \Aa(\uz(s,1);\uq) \). The same analysis of generators shows that the Hilbert series agree. 
\end{proof}
\subsection{} By Proposition \ref{prp:JMlimit} and by \cite[Theorem 2]{ryb_shift}, Lemma~\ref{lem:limit-bethe-algebras} implies that \( \lim_{t\to 1^{-}}\gamma_0(t) = \gamma_0(1) = \lim_{t\to1^{+}}\gamma_0(t) \), which proves the following Proposition.
\begin{proposition}
  \label{prp:gamma-continuous}
  The path \( \gamma_0 \) is continuous and thus homotopy equivalent to \( \gamma \) in \( \mathfrak{X}_{\uq} \).
\end{proposition}
\section{RSK from Inhomogeneous Gaudin Algebras}
\label{sec:RSK-from-bethe}
In this section we produce the RSK correspondence using parallel transport along collision paths. We freely use the notation from earlier sections of the paper.
\subsection{The case $n=2$} For the time being, let \( n=2 \). Let \( \uz(t) \) be a collision path and let \( \gamma(t) \in \mathfrak{X}_{\uq} \) be the associated path.  As noted in \ref{sec:collisionpathsdefn}, the algebra at \( \gamma(\infty) \) is \( \uBethe(\uq)^{\ox 2} \) and at $\gamma(0)$ is generated by \( \uJM_2 \) and \( \Delta\uBethe(\uq) \). We consider the action of the inhomogeneous Gaudin algebras \( \Aa(\uz(t);\uq) \) on the tensor product \( V_\lambda \ox V_\mu \) and the parallel transport induced along \( \gamma(t) \). The spectrum of \( \Aa(\uq)^{\ox 2} \) is \( \Ee_{\uq}(\lambda) \times \Ee_{\uq}(\mu) \). 
\subsection{} Since \( \uJM_2 \) commutes with \( \Delta U(\glr) \), the algebra at $\gamma(0)$, \( \uJM_2\cdot \Delta \Aa(\uq) \), acts on isotypic components of \( V_\lambda \otimes V_\mu \) which are of the form \( V_\nu \otimes (V_\lambda \otimes V_\mu)_\nu^{\sing}\). Here \( M^{\sing}_\nu \) denotes the highest weight vectors of weight \( \nu \) in a module \( M \). The subalgebra \( \uJM_2 \) acts on the first tensor factor and \( \Delta \Aa(\uq) \) on the second. We denote the spectrum of \( \uJM_2 \) on \( (V_\lambda \otimes V_\mu)_\nu^{\sing} \) by \( \Ee(\lambda,\mu)_\nu \). Thus the spectrum on this isotypic component is identified with \(  \Ee_{\uq}(\nu) \times \Ee(\lambda,\mu)_\nu  \).  By parallel transport we obtain a map 
\[ p^{\gamma}_{\lambda,\mu}: \Ee_{\uq}(\lambda) \times \Ee_{\uq}(\mu) \longrightarrow \bigsqcup_{\nu \in \Part_{\leq r}} \Ee_{\uq}(\nu) \times \Ee(\lambda,\mu)_\nu \]
\begin{theorem}[{\cite[Theorem 12.5]{HKRW}}]
  \label{thm:HKRW-mon-thm}
  The map \( p^{\gamma}_{\lambda,\mu} \) is an isomorphism of crystals.
\end{theorem}
\subsection{The general case} Now return to the situation for general \( n \). Identifying \( V_{k\varpi_1} \) with the space of homogeneous degree \( k \) polynomials in \( \CC[x_1,x_2,\ldots,x_r] \) we have a decomposition
\[ \CC[\Mat_{r\times n}] = \bigoplus_{\uk = (k_1,k_2,\ldots,k_n) \in \NN^n} V_{k_1\varpi_1} \ox V_{k_2\varpi_2} \ox \cdots \ox V_{k_n \varpi_1} \]
Fix a sequence \( \uk \in \NN^n \). We will consider the action of the inhomogeneous Gaudin algebras on the summand \(  V(\uk) := V_{k_1\varpi_1} \ox V_{k_2\varpi_2} \ox \cdots \ox V_{k_n \varpi_1} \). 
Again, the algebra at \( \gamma(\infty) \) is \(  \uBethe(\uq)^{\ox n} \) and at $\gamma(0)$ is generated by \( \uJM_n \) and \( \Delta^n\uBethe(\uq) \). The algebra at $\gamma(0)$ thus acts on isotypic components which are of the form \( V_{\nu} \otimes V(\uk)^{\sing}_\nu \). We denote the set of eigenspaces of the Jucys-Murphy elements acting on \( V(\uk)^{\text{sing}}_\nu \) by \( \Ee(\uk)_{\nu} \). Thus we obtain a map
\[ p_{\uk}^{\gamma}: \prod_{a=1}^n \Ee_{\uq}(k_a\varpi_1)  \longrightarrow \bigsqcup_{\nu \in \Part_{\le \min\{r,n\}}}  \Ee_{\uq}(\nu) \times \Ee(\uk)_{\nu}. \]
                     
\subsection{} The duality of Section~\ref{duality} identifies \( V(\uk)^{\sing}_\nu \) with a copy of \( (L_{\nu})_{\uk} \), the \( \uk \)-weight space in the irreducible \( \gl_n \)-module corresponding to the partition \( \nu \), denoted $L_{\nu}$. Furthermore \( \pi^r(\uJM_n) = \pi^n(\Delta^r\uGT_n) \) (see~\cite[Theorem~2]{Chervov:2010fz}). Together, these facts imply \( \Ee(\uk)_\nu = \Ee_{\underline{0}}(\nu)_{\uk}  \), the subset of \( \Ee_{\underline{0}}(\nu) \) consisting of eigenspaces contained in the \( \uk \)-weight spaces. Restriction of \( \comb^n_{\underline{0}} \) to \( \Ee_{\underline{0}}(\nu)_{\uk} \), produces a bijection \( \Ee(\uk)_\nu \longrightarrow \SSYT_n(\nu,\uk) \), the set of semistandard tableaux of shape \( \nu \) and content \( \uk \).
\subsection{} The spectrum \( \Ee_{\underline{\infty};\uq}(r\times n) \) of \( \Aa(\uq)^{\ox n} \) acting on \( \CC[\Mat_{r\times n}] \) has a decomposition
\[ \Ee_{\underline{\infty};\uq}(r\times n) = \bigsqcup_{\uk \in \NN^n} \prod_{a=1}^n \Ee_{\uq}(k_a\varpi_1). \]
The restriction of \( \comb_{\underline{\infty};\uq}^{r\times n} \) to \( \prod_{a=1}^n \Ee_{\uq}(k_a\varpi_1) \) induces a bijection \( \comb_{\underline{\infty};\uq}^{r\times n}(\uk): \prod_{a=1}^n \Ee_{\uq}(k_a\varpi_1) \longrightarrow \Mat_{r\times n}(\NN,\uk) \).
\begin{theorem}
  \label{thm:HKRW-to-RSK}
The map \( (\comb^r_{\underline{0}}\circ \kappa,\comb^n_{\underline{0}}) \circ p_{\uk}^{\gamma} \circ \comb_{\underline{\infty};\uq}^{r\times n}(\uk)^{-1} \) is the RSK correspondence restricted to \( \Mat_{r\times s}(\NN,\uk) \). 
\end{theorem}
\begin{proof}
  We proceed by induction. For \( n=2 \) this is obtained by an explicit calculation or by applying Theorem~\ref{thm:HKRW-mon-thm} and noting that \( V_{k_1\varpi_1} \ox V_{k_2 \varpi_1} \) is multiplicity free and the RSK correspondence is hence the unique morphism of crystals.
  For \( n > 2 \), we aim to prove the very outer square in the following figure commutes.
  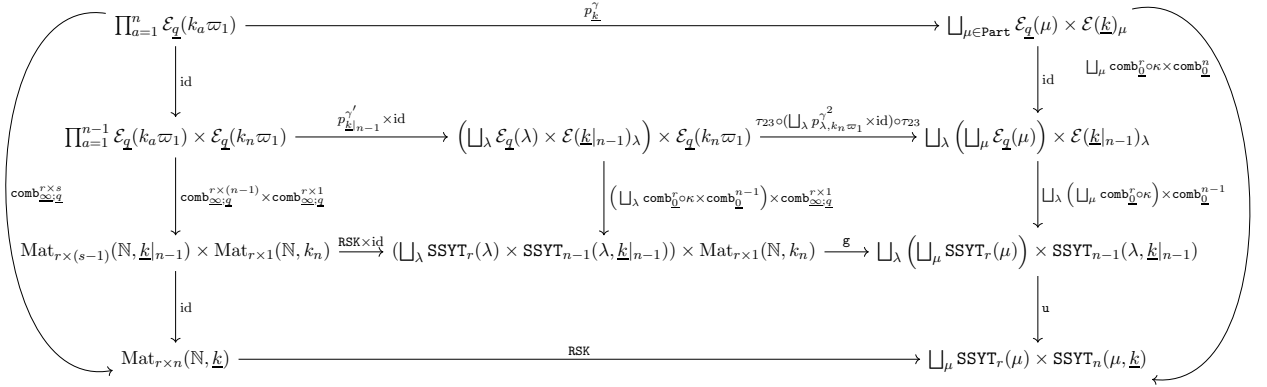
\begin{figure}[h!]
  
  \[ 
    \adjustbox{scale=0.65, left}{
  \begin{tikzcd}[row sep=huge]
    \prod_{a=1}^n \Ee_{\uq}(k_a\varpi_1)
        \arrow[d,"\id"]
        \arrow[ddd,bend right=100,"\comb^{r\times s}_{\underline{\infty};\uq}"]
        \arrow[rr,"p^{\gamma}_{\uk}"] &&
    \bigsqcup_{\mu \in \Part} \Ee_{\uq}(\mu) \times \Ee(\uk)_{\mu}
        \arrow[ddd,bend left=100,"\bigsqcup_{\mu}\comb^r_{\underline{0}} \circ \kappa \times \comb_{\underline{0}}^n "' near start]
        \arrow[d,"\id"]
    \\ 
    \prod_{a=1}^{n-1} \Ee_{\uq}(k_a\varpi_1) \times \Ee_{\uq}(k_n\varpi_1)
        \arrow[d,"\comb_{\underline{\infty};\uq}^{r\times (n-1)}\times \comb_{\underline{\infty};\uq}^{r\times 1}"]
        \arrow[r,"p^{\gamma'}_{\uk|_{n-1}} \times \id"] &
    \left( \bigsqcup_{\lambda} \Ee_{\uq}(\lambda) \times \Ee(\uk|_{n-1})_\lambda \right) \times \Ee_{\uq}(k_n\varpi_1)
        \arrow[d,"\big( \bigsqcup_\lambda\comb_{\underline{0}}^{r} \circ \kappa \times \comb_{\underline{0}}^{n-1} \big) \times \comb_{\underline{\infty};\uq}^{r\times 1}"]
        \arrow[r,"\tau_{23}\circ (\bigsqcup_{\lambda} p^{\gamma^2}_{\lambda,k_n\varpi_1} \times \id) \circ \tau_{23}"] &
    \bigsqcup_{\lambda} \left( \bigsqcup_{\mu} \Ee_{\uq}(\mu) \right) \times \Ee(\uk|_{n-1})_\lambda
        \arrow[d,"\bigsqcup_\lambda \big( \bigsqcup_\mu \comb_{\underline{0}}^r\circ \kappa \big) \times \comb_{\underline{0}}^{n-1}"]
    \\
    \Mat_{r \times (s-1)}(\NN,\uk|_{n-1}) \times \Mat_{r \times 1}(\NN,k_n)
        \arrow[d,"\id"]
        \arrow[r,"\RSK \times \id"] &
    \left( \bigsqcup_{\lambda} \SSYT_r(\lambda) \times \SSYT_{n-1}(\lambda,\uk|_{n-1}) \right) \times \Mat_{r \times 1}(\NN,k_n)
        \arrow[r,"\mathtt{g}"] &
    \bigsqcup_{\lambda} \left( \bigsqcup_{\mu} \SSYT_r(\mu) \right)\times \SSYT_{n-1}(\lambda,\uk|_{n-1})
        \arrow[d,"\mathtt{u}"]
    \\
    \Mat_{r\times n}(\NN,\uk)
        \arrow[rr,"\RSK"] &&
    \bigsqcup_{\mu} \SSYT_r(\mu) \times \SSYT_n(\mu,\uk)
  \end{tikzcd}
  }
  \]
  \caption{Induction diagram}
  \label{fig:induction-diagram}
\end{figure}
 It is Figure~\ref{fig:induction-diagram} that displays how our induction will work. Let \( \uk|_{n-1} = (k_1,k_2,\ldots,k_{n-1}) \). By Proposition~\ref{prp:gamma-continuous}, \( \gamma_0 \) is homotopy equivalent to \( \gamma \). If we let \( p_{b,a}^{\gamma_0} \) be the map induced by parallel transport along \( \gamma_0 \) from \( t=a \) to \( t=b \) we thus have \( p_{\uk}^{\gamma} = p_{0,\infty}^{\gamma_0} = p_{0,1}^{\gamma_0} \circ p_{1,\infty}^{\gamma_0}  \). Let \( \uz'(t) = (z_1(t),z_2(t),\ldots,z_{n-1}(t)) \) be the collision path given by only considering the first \( n-1 \) points, and denote by \( \gamma'(t) \) the associated path in the moduli space. According to Lemma~\ref{lem:limit-bethe-algebras}, the path of algebras along $\gamma_0(t)$ for $t\geq 1$ equals  \(  \Aa(\uz'(t-1);\uq) \otimes \Aa(\uq) \). Thus we have
  \[ p_{1,\infty}^{\gamma_0} = p_{\uk|_{n-1}}^{\gamma'}\times \id : \prod_{a=1}^{n-1} \Ee_{\uq}(k_a\varpi_1) \times  \Ee_{\uq}(k_n\varpi_1) \longrightarrow \left( \bigsqcup_{\lambda \in \Part_{\le \min\{r,n-1\}}} \Ee_{q}(\lambda) \times \Ee(\uk|_{n-1})_{\lambda} \right) \times \Ee_{\uq}(k_n\varpi_1). \]
We will apply the induction hypothesis to this factorisation and the map \( p^{\gamma'}_{\uk|_{n-1}} \).
We must also consider the parallel transport map \( p^{\gamma_0}_{0,1} \). This comes from moving just two particles, one at \( 0 \) and the other at \( \frac{2-t}{1-t}z_n(t) \) while the rest remain fixed. The path \( \uz^2(t) = (0,\frac{2-t}{1-t}z_n(t)) \) is not a collision path, firstly \( t \in (0,1) \) and secondly since the first coordinate does not tend to infinity. Shifting both coordinates by \( z_{n-1}(t) \) and reparametrising 
gives a collision path, without effecting the induced path in the moduli space and the algebras involved. Let \( \gamma^2 \) be the path in the moduli space induced by \( \uz^2(t) \).
Let \( \lambda \in \Part \) and \( l\in \mathbb{N} \). By the above parallel transport \( p^{\gamma^2}_{\lambda,l\varpi_1} : \Ee_{\uq}(\lambda) \times \Ee_{\uq}(l\varpi_1) \longrightarrow \bigsqcup_{\mu} \Ee_{\uq}(\mu) \times \Ee(\lambda, l\varpi_1)_\mu \) arises from a collision path and is thus by Theorem~\ref{thm:HKRW-mon-thm} a morphism of crystals. Let \( \Part(\lambda,l) \) be the set of partitions formed by adding \( l \) boxes to \( \lambda \), no two in the same column. The Pieri rule states that the space \( (V_\lambda \otimes V_{l\varpi_1})^{\sing}_\mu \) is zero unless \( \mu \in \Part(\lambda,l) \) in which case it is one dimensional. Thus \( \Ee(\lambda, l\varpi_1)_{\mu} \) is a single point and we can identify the codomain of \( p^{\gamma^2}_{\lambda,l\varpi_1} \) with \( \bigsqcup_{\mu \in \Part(\lambda,l)} \Ee_{\uq}(\mu) \) where the union is over \( \mu \in \Part(\lambda,l) \). Using this fact and Lemma~\ref{lem:limit-bethe-algebras}, we have
  \begin{align*}
    p_{1,0}^{\gamma_0} = \tau_{23}\circ \left(\bigsqcup_{\lambda} p^{\gamma^2}_{\lambda,k_n\varpi_1} \times \id\right) \circ \tau_{23} :& \left( \bigsqcup_{\lambda \in \Part_{\le \min\{r,n-1\}}} \Ee_{\uq}(\lambda) \times \Ee(\uk|_{n-1})_\lambda \right) \times \Ee_{\uq}(k_n\varpi_1) \\ &\hspace{32pt}\longrightarrow \bigsqcup_{\lambda \in \Part_{\le \min\{r,n-1\}}} \left( \bigsqcup_{\mu \in \Part(\lambda,k_n)} \Ee_{\uq}(\mu) \right) \times \Ee_{\uq}(k_n\varpi_1),
  \end{align*}
  where \( \tau_{23} \) is simply the map that swaps the second and third factors. This explains the second row and the commutativity of the top square of Figure~\ref{fig:induction-diagram}.
  
  Let \( A \in \Mat_{r \times n}(\NN,\uk)  \). We can think of \( A \) as a pair \( (A_{\le r,\le n-1}, A_{\le r,n}) \) (the leftmost \( r \times (n-1) \) submatrix and the final column). With this identification, \( \Ee_{\uq}(r\times n) = \Ee_{\uq}(r\times n-1)\times \Ee_{\uq}(r\times 1) \). The map \( \comb_{\underline{\infty};\uq}^{r \times s} \) factors through this identification via the map \( \comb_{\underline{\infty};\uq}^{r \times (n-1)} \times \comb_{\underline{\infty};\uq}^{r \times 1} \) which explains the commutativity of the leftmost cell of Figure~\ref{fig:induction-diagram}.
  The algebra \( \uJM_{n} \subset U(\glr)^{\ox n} \) is generated by \( \uJM_{n-1}\ox 1 \) and \( \uJM_2^{^{(1\ldots n-1)(n)}} \). We also have that \( \pi^r(\uJM_n) = \pi^n(\Delta^r\uGT_n) \) and \( \pi^r(\uJM_{n-1}\ox \id) = \pi^n(\Delta^r\uGT_{n-1}) \). Observe that a point in the spectrum of \( \uGT_n \) is completely determined by a point in the spectrum of \( \uGT_{n-1} \), and a highest weight. Thus we obtain an identification
  \[ \bigsqcup_{\mu \in \Part_{\le \min\{r,n\}}} \Ee(\uk)_{\mu} = \bigsqcup_{\lambda \in \Part_{\le \min\{r,n-1\}}} \bigsqcup_{\mu \in \Part(\lambda,k_n)} \Ee(\uk|_{n-1})_\lambda. \]
  In a similar fashion, if \( \mu \in \Part(\lambda,k_n) \) then given \( T \in \SSYT_{n-1}(\lambda,k|_{n-1}) \) there is a unique tableau \( T' \in \SSYT_n(\mu,k_n) \) given by simply adding boxes containing \( n \) to
  \( T \) in the unique way determined by the shape \( \mu \). This induces a \( \glr \)-crystal morphism
  \[ \mathtt{u}: \bigsqcup_{\lambda \in \Part_{\le \min\{r,n-1\}}} \SSYT_r(\mu) \times \SSYT_{n-1}(\lambda,\uk|_{n-1}) \longrightarrow \bigsqcup_{\mu \in \Part_{\min\{r,n\}}} \SSYT_r(\mu) \times \SSYT_n(\mu,\uk). \]
  The right hand square in the diagram commutes by the definition of \( \comb_{\underline{0}}^n \) and since this is compatible with restriction to \( \uGT_{n-1} \).
  Recall the set \( \Mat_{r \times 1}(\NN,k_n) \) has the structure of a \( \glr \)-crystal corresponding to the module \( V_{k_n\varpi_1} \). By the Pieri rule, the tensor product of crystals \( \SSYT_r(\lambda) \times \Mat_{r\times 1}(\NN,k_n) \) is isomorphic to \( \bigsqcup_{\mu \in \Part(\lambda,k_n)}\SSYT_r(\mu) \). The unique crystal isomorphism is given by sending a pair \( (T,A) \) to the result of inserting \( 1 \) into \( T \) exactly \( A_{11} \) times, then inserting \( 2 \) exactly \( A_{22} \) times, etc. This defines a crystal isomorphism
  \begin{align*}
    \mathtt{g}: &\left( \bigsqcup_{\lambda \in \Part_{\le \min\{r,n-1\}}} \SSYT_r(\lambda) \times \SSYT_{n-1}(\lambda,\uk|_{n-1}) \right) \times \Mat_{r\times 1}(\NN,k_n) \\
    &\hspace{120pt}\longrightarrow \bigsqcup_{\lambda \in \Part_{\le \min\{r,n-1\}}} \bigsqcup_{\mu \in \Part(\lambda,k_n)} \SSYT_r(\mu) \times \SSYT_{n-1}(\lambda,k|_{n-1}).
  \end{align*}
  The bottom most rectangle in Figure~\ref{fig:induction-diagram} commutes then by Proposition~\ref{prp:rsk-rigidity}.
  Now we can note that the middle left square commutes by induction, and the only remaining thing to show is that the middle right square commutes. We note that everything in sight is a \( \glr \)-crystal isomorphism by Proposition \ref{prp:GT-HKRW-crystals-agree}. Thus the images in the first factor of the product agree. The fact that the images in the second factor agree follows since projection of either map onto the second factor is \( \comb_{\underline{0}}^{n-1} \).
\end{proof}
As a corollary, we obtain Theorem~\ref{thm:main}.
\begin{corollary}
  \label{cor:main-theorem}
For any \( A \in \Mat_{r\times n}(\NN) \), \( S(A)=P(A) \) and \( T(A)=Q(A) \).
\end{corollary}
\begin{proof}
First note that by definition \( S = \comb_{\underline{0}}^r\circ \kappa \circ \bigsqcup_{\uk} p^\gamma_{\uk} \circ \left( \comb_{\underline{\infty};\uq}^{r\times n} \right)^{-1} \). Thus by Theorem~\ref{thm:HKRW-to-RSK}, \( S(A)=Q(A) \).
To show that \( T(A)=Q(A) \) we swap the roles of \( r \) and \( n \) by appealing to \( (\glr, \gl_n) \)-duality. By choosing a collision path \( \uq \) and fixing a weight \( \underline{l} = (l_1,l_2,\ldots,l_r) \), we obtain from Theorem~\ref{thm:HKRW-to-RSK} a map
\[ p_{\underline{l}}^{\gamma}: \prod_{i=1}^r \Ee_{\uz}(l_i\varpi_1)  \longrightarrow \bigsqcup_{\nu \in \Part_{\le \min\{r,n\}}}  \Ee_{\uz}(\nu) \times \Ee(\underline{l})_\nu, \]
such that \( P = \comb_{\underline{0}}^n\circ \kappa \circ \bigsqcup_{\underline{l}} p_{\underline{l}}^{\gamma} \circ \left( \comb_{\underline{\infty};\uz}^{n \times r} \right)^{-1} \). On the other hand \( T = \comb_{\underline{0}}^n \circ \kappa \circ \bigsqcup_{\underline{l}} p_{\underline{l}}^{\gamma} \circ \left( \comb_{\underline{\infty};\uq}^{r \times n} \right)^{-1} \)
Note however that by construction \( \comb_{\underline{\infty};\uq}^{r \times n}\circ\left( \comb_{\underline{\infty};\uz}^{n \times r} \right)^{-1}(A) = A^t \), the transpose map. Thus \( T(A) = P(A^t) = Q(A) \) (by Theorem~\ref{thm:RSK}). This completes the proof.
\end{proof}
  
\section{Cherednik algebras and Calogero-Moser cells}
\subsection{}
\label{sec:RCA}
Let $\bc$ be a variable. The rational Cherednik algebra of $\mathfrak{S}_n$ is the quotient $\CH$ of $\CC[\mathfrak{S}_n]\ltimes \CC\langle \bc, x_1,\dots,x_n,y_1,\dots,y_n\rangle$ by the relations that $\bc$ is central and
\begin{align*}
[x_i,x_j]& = 0, & [y_i,y_j]&=0,\\
[y_i,x_j]&= \bc \, (i,j), & [x_i,y_i]&=- \bc \sum_{j\neq i} (i,j),
\end{align*} where $1\leq i, j \leq n$ and $i\neq j$.
Let $Z$ be the centre of $\CH$. There is an inclusion 
\[
	\iota:P:=\CC[\bc]\ot\CC[x_1,\dots,x_n,y_1,\dots,y_n]^{\mathfrak{S}_n\times \mathfrak{S}_n} \hookrightarrow Z.
\]
Let $V=\CC^n$ and $X=\spec Z$. The inclusion $\iota$ induces a surjection 
\[
	\Upsilon:X \longrightarrow \CC \times \CC^n/\mathfrak{S}_n \times \CC^n/\mathfrak{S}_n 
\]
\subsection{}For $c \in \CC$ let $P_c, Z_c, \CH_c$ and $\Upsilon_c$ be the specialisations of $P,Z, \CH$ and $\Upsilon$ at $\bc=c$.
For any non-zero $c\in \CC$ the preimage of $\{c\} \times  \CC^n/\mathfrak{S}_n \times \CC^n/\mathfrak{S}_n$ is identified with the Calogero--Moser space
\[
	\CM_n=\{(A,B)\in \mf{gl}_n \times \mf{gl}_n,\ \rk([A,B]+\id)=1\}/GL_n
\]
and $\Upsilon_c$ with the map that send pairs of matrices to their unordered set of eigenvalues.
The preimage at $c=0$ is identified with equivalences classes of pairs of commuting matrices, hence with $(\CC^n\times \CC^n)/\Delta\mathfrak{S}_n$ where $\Delta \mathfrak{S}_n$ denotes the diagonal of $\mathfrak{S}_n$ in $\mathfrak{S}_n\times \mathfrak{S}_n$.
\subsection{} The work of Mukhin--Tarasov--Varchenko relates the spectrum of a Bethe algebra to the rational Cherednik algebra of type $\mathfrak{S}_n$. We recall this now. 
Let $V = \CC^n$ and $\cV=V^{\otimes n}[\buz,\buq]$. Let $(V^{\otimes n})_{\mathbbm{1}}$ be the $(1,1,\dots ,1)$ weight subspace of $V^{\ot n}$ and $\cV_{\mathbbm{1}}=(V^{\otimes n})_{\mathbbm{1}}[\buz,\buq]$. The algebra $U(\mf{gl}_n[t])[\buq]$ acts on $\cV$ as explained in \cite[Section 2.4]{Mukhin:2014ga}, varying in $\buz$. 
\subsection{}
The universal Bethe algebra, $\mathsf{B}_n$ is a commutative subalgebra of $U(\mf{gl}_n[t])[\buq]$, see \cite[Section 2.2]{Mukhin:2014ga}. Through the action of $U(\mf{gl}_n[t])[\buq]$ on $\cV$, $\mathsf{B}_n$ specialises to a commutative algebra $\overline{\mathsf{B}}_n$ in $\End(\cV_{\mathbbm{1}})$. In turn, this specialises for $(\uz,\uq)\in \Creg^n\times\Creg^n$ to the inhomogeneous Gaudin algebra $\mathcal{A}(\underline{z}; \underline{q}) $ introduced in Subsection~\ref{defn:iga}, see \cite[Corollary~1]{ryb_uniqueness}.
\subsection{}
The space $(V^{\ot n})_{\mathbbm{1}}$ can be identified with $\CC[\mathfrak{S}_n]$ via
\[
	w \mapsto e_{w}:=e_{w(1)} \ot \dots \ot e_{w(n)}.
\]
The PBW decomposition of $\CH_1$,  
$\CC[x_1, \ldots , x_n]\ot \CC[\mathfrak{S}_n] \ot \CC[y_1, \ldots , y_n] \longrightarrow \CH$
then induces a $\CC$-linear isomorphism $\alpha: \cV_{\mathbbm{1}}\ \rightarrow \CH_1$ given by
\[
e_w f(z_1, \ldots, z_n)g(q_1, \ldots , q_n) \longmapsto f(x_1, \ldots ,x_n)\otimes w \otimes g(y_1, \ldots , y_n).
\]
The key results of Mukhin-Tarasov-Varchenko that relate the rational Cherednik algebra to the work in the earlier part of the paper are the following.
\begin{theorem}
\label{thm:MTVportmanteau}
	\begin{enumerate}
		\item  Under the map $\alpha$, the action of the centre $Z_1$ of $\CH_1$ by left multiplication is identified with the Bethe algebra $\uB_n$ acting on $\cV_{\mathbbm{1}}$, \cite[Theorem 2.8]{Mukhin:2014ga}. This induces an algebra isomorphism
\[
	\beta: \uB_n \longrightarrow Z_1, 
\]
\cite[Corollary 2.9]{Mukhin:2014ga}.
\item The subalgebra $\CC[\buz,\buq]^{\mathfrak{S}_n\times \mathfrak{S}_n}$ of $\End(\cV_{\mathbbm{1}})$ is contained in $\uB_n$, \cite[Lemma 2.6]{Mukhin:2014ga}, and the isomorphism $\beta$ restricts to the tautological one
\[
	\CC[\buz,\buq]^{\mathfrak{S}_n\times \mathfrak{S}_n} \longrightarrow P, 
\]
\cite[Theorem 4.3]{Mukhin:2014ga}.
\item The action of $\uB_n$ on $\cV_{\mathbbm{1}}$ commutes with the scalar action of $\CC[\bz,\bq]$, \cite[Lemma 2.3]{Mukhin:2014ga}.
	\end{enumerate}
\end{theorem}
\subsection{} By the above,  we have that \[
\Upsilon_1:	\spec \uB_n\longrightarrow \CC^n/\mathfrak{S}_n \times \CC^n/\mathfrak{S}_n.
\] 
Let $\sigma_1, \ldots , \sigma_n$ be the symmetric functions such that $$\prod_{i=1}^n (u-{\bf{z}}_i) = \sum_{i=1}^n (-1)^i \sigma_i (\buz) u^{n-i}.$$
\begin{theorem}
\label{thm:bethegaudin}
\begin{enumerate}
    \item Let $\underline{z} = (z_1, \ldots, z_n)\in \Creg^n$ and $\underline{q} = (q_1, \ldots , q_n)\in \Creg^n$. The algebra $\uB_n/\langle \sigma_i(\buz) - \sigma_i(\underline{z}), \sigma_i(\buq) - \sigma_i(\underline{q}): 1\leq i \leq n\rangle$ is isomorphic to the image of $\mathcal{A}(\underline{z}; \underline{q})$ in $End((V^{\ot n})_1)$.
    \item The morphism $\Upsilon_1$ is unramified over the image of $\Rdel^n \times \Rdel^n$ in $\CC^n/\mathfrak{S}_n \times \CC^n/\mathfrak{S}_n.$ 
    \item  Let $\underline{q} = (q_1, \ldots , q_n)\in \Rdel^n$. The algebra $\uB_n/\langle \sigma_i(\buz), \sigma_i(\buq) - \sigma_i(\underline{q}): 1\leq i \leq n\rangle$ contains the image of $\mathcal{A}(\underline{0}; \underline{q})$ in $End((V^{\ot n})_{\mathbbm{1}})$.
\end{enumerate}
\end{theorem}
\begin{proof}
Part (1) is a consequence of \cite[Lemma 5.4]{Mukhin:2014ga} and (2) is proved by Mukhin-Tarasov-Varchenko in \cite[Corollary 7.4]{Mukhin:2008kj}. For (3) the algebra $\mathcal{A}(0;\underline{q})$ acts via $\mathcal{A}(q)$ on $(V^{\ot n})_{\mathbbm{1}}$, by Proposition \ref{prp:GT-limit}. This action is semisimple, as explained in Subsection \ref{sec:GTdeg} and is generated by the 'classical' Hamiltonians, thanks to \cite[Corollary 3.4]{Mukhin:2010ky} and the duality of \cite[Theorem 3.1]{Mukhin:2009in}. These elements are the evaluation at $\underline{z} = \underline{0}$ of the dynamical Casimir Hamiltonians $$\nabla_i (\uz, \uq) = \sum_{k=1}^{n} z_kE_{ii}^{(k)} +  \sum_{j\neq i} \frac{\kappa_{ij}}{q_i-q_j} $$ where $\kappa_{ij} = 2(E_{ij}E_{ji} + E_{ji}E_{ij})\in U(\mathfrak{gl}_n).$ Since these elements belong to $\mathcal{A}(\underline{z}; \underline{q})= \uB_n/\langle \sigma_i(\buz) - \sigma_i(\underline{z}), \sigma_i(\buq) - \sigma_i(\underline{q}): 1\leq i \leq n\rangle$ for $\underline{z}\in \Creg^n$ by \cite[Proposition 9.5(3)]{HKRW}, it follows that their limit at $\underline{z} = \underline{0}$ belongs to $\uB_n/\langle \sigma_i(\buz), \sigma_i(\buq) - \sigma_i(\underline{q}): 1\leq i \leq n\rangle$, as required.  
\end{proof}
\subsection{} 
Now we move to the setup of \cite[Chapter 6,  and Appendix B]{BonRouq}. 
 Let $K,L$ be the fraction fields  of $P$ and $Z$ respectively, and let $R$ be the integral closure of $Z$ inside a Galois closure $F$ of the extension $L/K$. We denote by $\rho$ the projection
\[
\rho:	\spec R \longrightarrow \spec P.
\]
Let $\mf p_0$ be the ideal in $P$ generated by $\bc$ and let $\mf r_0$ an ideal in $R$ lying over $\mf p_0$. Fix an isomorphism $R_0=R/\mf r_0\rightarrow \CC[V\times V]$ that is the identity on $Z_0$. Let $(\underline{z},\underline{q})$ be a generic point in $\CC^n\times \CC^n$ and let $y_0$ be its preimage in $\spec R_0$. Let $\rho(y_0) = ([\uz], [\uq])$, the image of $(\underline{z}, \underline{q})$ in $\CC^n/\mathfrak{S}_n \times \CC^n/\mathfrak{S}_n$. We obtain a bijection
\[
\mathfrak{S}_n \longrightarrow \Upsilon_0^{-1}(([\uz],[\uq]))
\]
given by $w \mapsto (w(\underline{z}),\underline{q})\Delta \mathfrak{S}_n$.
Let $\gamma$ be a path in $\CC\times \CC^n/\mathfrak{S}_n\times \CC^n/\mathfrak{S}_n$ such that $\gamma(0)=(0,x_0)$ and such that $\rho$ is unramified over $\gamma(t)$ for $t\in [0,1)$. For any $w \in \mathfrak{S}_n$ there is a unique path $\gamma_w$ in $\spec Z$ lifting $\gamma$ and such that $\gamma_w(0)=(0,(w(\underline{z}),\underline{q})\Delta \mathfrak{S}_n)$.
\begin{definition} \cite[Definition 6.6.1]{BonRouq}
Two elements $w,w'$ are in the same Calogero-Moser $\gamma$-cell if $\gamma_w(1)=\gamma_{w'}(1)$.
\end{definition}
\subsection{} 
\label{sec:gamma} Assume that $\underline{z}\in \Rdel^n$ and $\underline{q}\in \Rdel^n$. Let $\tilde \gamma:[0,1]\rightarrow \CC \times \CC^n \times \CC^n$ be the path defined by $\tilde \gamma(t)=(t,(1-t)\uz,\uq)$. Let $\gamma$ be the projection of $\tilde \gamma$ on $\CC\times \CC^n/\mathfrak{S}_n\times \CC^n/\mathfrak{S}_n$. By Theorem \ref{thm:bethegaudin}(2) above and \cite[Lemma 6.3.4]{BonRouq}, $\gamma([0,1))$ lies in the unramified locus of $\spec P = \CC \times \CC^n/\mathfrak{S}_n \times \CC^n/\mathfrak{S}_n$. Therefore we can use this path to define Calogero-Moser $\gamma$-cells. 
There is a unique lift $\tilde \gamma$ of $\gamma$ in $\spec R$ such that $\tilde \gamma(0)=y_0$. Since $\gamma(1)$ is a general point in $\{1\}\times \{ \underline{0} \} \times \CC^n/\mathfrak{S}_n $, there is a unique prime ideal $\mf r$ of $R$ such that $y_1=\tilde\gamma(1)$ lives in the irreducible component of $\rho^{-1}(\{1\} \times \{ \underline{0} \} \times \CC^n/\mathfrak{S}_n)$ determined by $\mf r$. After perturbing $z$, we can assume $y_1$ lives in the maximal open subset $O$ of the irreducible component whose points have stabiliser the inertia group $I_{\mf r}$, the inertia group of $\mathfrak{r}$ in $\mathsf{Gal}(F/K)$. It then follows from \cite[Proposition 6.6.2 and comments after Lemma B.7.2]{BonRouq} that the Calogero-Moser $\gamma$-cells are the same as the right Calogero-Moser cells (with respect to $\mf r$) of $\mathfrak{S}_n$, as defined at the beginning of \cite[Part III]{BonRouq}. 
\subsection{} We are now able to give the main application of our results.
\begin{theorem}\label{thm:CMEqualKL}
The right Calogero-Moser cells for $\mathfrak{S}_n$ agree with the right Kazhdan--Lusztig cells for $\mathfrak{S}_n$, both being described by $w\sim^R w'$ if and only if $P(w) = P(w')$.
\end{theorem}
\begin{proof} 
Thanks to the above discussion, we must determine precisely when $\gamma_w(1) = \gamma_{w'} (1)$. By Theorem \ref{thm:bethegaudin}(3) the fibres of $\Upsilon: \spec Z \rightarrow \spec P$ above the points in $\gamma((0,1))$ are precisely the elements in the $(1,\ldots , 1)$-weight space in the sets $\mathcal{E}_{s\underline{z}, \underline{q}}(\underline{1})$ where $s = (1-t)/t$.    
For any $s \in \CC^*$ and any $(s\uz,\uq)\in \Creg^n\times \Creg^n$ there are $n!$ choices of $(p_1, \ldots , p_n)\in \mathbb{C}^n$ such that matrices 
\begin{align}\label{pairmat}
       	Y& = \begin{pmatrix} p_1 & s^{-1}(z_1-z_2)^{-1} & \cdots & s^{-1}(z_1-z_n)^{-1} \\ s^{-1}(z_2-z_1)^{-1} & p_2 & \cdots & s^{-1}(z_2-z_n)^{-1} \\ \vdots & \vdots & & \vdots \\ s^{-1}(z_n-z_1)^{-1} & s^{-1}(z_n-z_2)^{-1} & \cdots & p_n  \end{pmatrix},
\end{align}
have eigenvalues $[\underline{q}]$. The set $\{(Z = \diag(sz_1, \ldots , sz_n), Y) \in \CM_n: Y \text{ as in } \eqref{pairmat}\} = \Upsilon_1^{-1}([s\uz], [\uq])$. As $t\rightarrow 0$, $s= (1-t)/t \rightarrow \infty$ and so the matrices tend to $Y = \diag ( p_1, \ldots , p_n)$ where $[\underline{p}] = [\uq]$. In the inhomogeneous Gaudin algebra description the matrix $\diag (q_{w^{-1}(1)}, \ldots q_{w^{-1}(n)})$ corresponds to the eigenbasis element $x_{w} = \prod_{i,j}x_1^{A_{11}}\cdots x_r^{A_{r1}}\otimes \cdots \otimes x_1^{A_{1n}}\cdots x_r^{A_{rn}} \in (V^{\ot n})_{\mathbbm{1}}$ from \ref{sec:ebasis}, where $A_{ij} = \delta_{j w(i)}$ for $1\leq i,j \leq n$. If we identify these matrices then with $\mathfrak{S}_n$ via $w \mapsto \diag (q_{w^{-1}(1)}, \ldots q_{w^{-1}(n)})$ we therefore recover the labelling in  \ref{sec:ebasis} of the $(1,\ldots , 1)$-weight space in $\mathcal{E}_{{\infty}, \underline{q}}(\underline{1})$ by permutations. This corresponds to $(0, [\uz, w^{-1}\uq]\Delta S_n) = (0, [w(\uz), \uq]) \in \{0\} \times (\CC^n\times \CC^n)/\Delta\mathfrak{S}_n$, and so is labelled by $w$ in the Galois-theoretic labelling. Hence the two labellings, one from inhomogeneous Gaudin algebras and the other from Galois theory, agree.  
The results of \ref{sec:GTdeg} and Theorem \ref{thm:main} describe the continuation along the path $\gamma$ of the elements $x_{w} \in \mathcal{E}_{\infty, \underline{q}}(\underline{1})$ for $w\in \mathfrak{S}_n$. In the limit at $t=1$ $x_w$ and $x_{w'}$ will  have the same eigenvalues for the action of $\mathcal{A}(\underline{0},\underline{q})$ if and only if $P(w) = P(w')$.  On the other hand for general $\uq\in \Creg^n$, by \cite[Theorem 1.5]{Mukhin:2012ee}, the subalgebra $\mathcal{A}(\underline{0};\underline{q})$ of $\uB_n/\langle \sigma_i(\buz), \sigma_i(\buq) - \sigma_i(\underline{q}): 1\leq i \leq n\rangle$ determines all the distinct closed points of the fibre, with $\sum_{\lambda\in \Part (n)} |\SYT (\lambda)|$ such points.  
Since $\uq\in \Rdel^n$ is a general condition, this shows that $w$ and $w'$ belong to the same Calogero-Moser $\gamma$-cell if and only if $P(w) = P(w')$.
\end{proof}
\subsection{} We can also describe the left CM-cells and the two-sided cells. 
\begin{corollary} 
\begin{enumerate}
    \item 
The left Calogero-Moser cells for $\mathfrak{S}_n$ agree with the right Kazhdan--Lusztig cells for $\mathfrak{S}_n$, both being described by $w\sim^L w'$ if and only if $Q(w) = Q(w')$. 
\item
The two-sided Calogero-Moser cells for $\mathfrak{S}_n$ agree with the two-sided Kazhdan--Lusztig cells for $\mathfrak{S}_n$, both being described by $w \sim^{LR} w'$ if and only the partition underlying $P(w)$ is the same as the partition underlying $P(w')$.
\end{enumerate}
\end{corollary}
\begin{proof}
  (1) For the left cells we use the path $\tilde{\lambda}: [0,1] \rightarrow \mathbb{C} \times \CC^n \times \CC^n$ given by $\tilde{\lambda}(t) = (t, \underline{z}, (1-t)\underline{q})$, with $\uz$ and $\uq$ as before. We can then repeat the argument of Theorem \ref{thm:CMEqualKL}.
  
  Alternatively let the projection of $\tilde{\lambda}$ to $\mathbb{C}\times \CC^n/\mathfrak{S}_n \times \CC^n/\mathfrak{S}_n$ be denoted by $\lambda$. As explained in \cite[Proposition 5.6.1]{BonRouq} there is an automorphism $\sigma_H$ of $R$ which restricts to automorphisms of $P$ and $Q$, which are induced from the mapping $$(c, \uq, \uz) \mapsto (c, -\uq, \uz).$$ Note that $-\uq$ is no longer an element of $\Rdel^n$, rather $w_0 (\-uq)$ is, where $w_0 = (1 \,n) (2 \, n-1) \cdots \in \mathfrak{S}_n$ is the longest word. The mapping $\sigma_H \circ \lambda$ is just the mapping $\gamma$ introduced in \ref{sec:gamma}. 
  
  We need to understand when the lifts of $\lambda$ to $\spec P$ beginning at the points $(0, (w(\uz), \uq)\Delta \mathfrak{S}_n)$ and $(0, (w'(\uz), \uq)\Delta \mathfrak{S}_n)$ collide at the endpoint. This happens if and only if it happens on applying $\sigma_H$. In other words, this happens if and only if the lifts of $\gamma$ to $\spec P$ beginning at the points $\sigma_H ((0, (w(\uz), \uq)\Delta \mathfrak{S}_n))$ and $\sigma_H(0, (w'(\uz), \uq)\Delta \mathfrak{S}_n)$ collide at the endpoint. But 
  \begin{eqnarray*}
  \sigma_H ((0, (w(\uz), \uq)\Delta \mathfrak{S}_n)) &=& (0, (-\uq, w(\uz))\Delta \mathfrak{S}_n)\\ &=&(0, (w^{-1} (-\uq), \uz)\Delta \mathfrak{S}_n) \\&=& (0, (w^{-1} w_0 (\uq), \uz)\Delta \mathfrak{S}_n) \end{eqnarray*} where, up to homotopy, we can assume that $w_0 (\uq) = - \uq$. Thus $\lambda_{w}(1)  = \lambda_{w'} (1) $ if and only if $\gamma_{w^{-1}w_0}(1) = \gamma_{w'^{-1}w_0}(1)$. By Theorem \ref{thm:CMEqualKL} this happens if and only if $P(w^{-1}w_0) = P(w'^{-1}w_0)$. Equivalently $P((w_0w)^{-1}) = P((w_0w')^{-1})$. Since $P(\tau^{-1}) = Q(\tau)$ this is the same as $Q(w_0w ) = Q(w_0w')$. By \cite[A2.1.11]{Stanley:1999eh} $Q(w_0w) = (\textsf{evac}Q(w))^t$ where $\textsf{evac}$ is the evacuation procedure on tableaux, a bijection. It follows that $Q(w_0w ) = Q(w_0w')$ if and only if $Q(w) = Q(w')$, as claimed. 
  
  (2) Choosing the path $\mu (t) = (t, (1-t)\uz, (1-t)\uq)$ ensures that the two-sided cells are unions of left cells and also right cells. Thus if $w$ and $w'$ give rise to the same partition, then we can find $y\in\mathfrak{S}_n$ such that $P(w) = P(y)$ and $Q(y) = Q(w')$. It follows that $w$ and $w'$ are in the same two-sided call. On the other hand, thanks to \cite[Theorem 14.4.1]{BonRouq}, the two-sided cells have cardinality the squares of the dimensions of the irreducible representations of $\mathfrak{S}_n$. Hence they cannot have any more elements in them, confirming the final claim.  
\end{proof}
\printbibliography

\end{document}